\documentclass[a4paper,11pt]{article}

\usepackage{authblk}

\usepackage[utf8]{inputenc} 
\usepackage[T1]{fontenc}    
\usepackage{indentfirst}
\usepackage{lmodern}

\usepackage{mathtools}
\usepackage{verbatim}
\usepackage{url}            
\usepackage{booktabs}       
\usepackage{amsfonts}       
\usepackage{amsmath}
\usepackage{amssymb}
\usepackage{amsthm}
\usepackage{microtype}      
\usepackage[a4paper]{geometry} 
\usepackage{enumitem}
\usepackage{xcolor}

\usepackage[linktocpage=true]{hyperref}
\hypersetup{
    colorlinks=true,
    linkcolor=blue,
    urlcolor=yellow
}

\usepackage{graphicx}
\usepackage[makeroom]{cancel}
\usepackage{tikz}

\newcommand*\circled[1]{\tikz[baseline=(char.base)]{
            \node[shape=circle,draw,inner sep=2pt] (char) {#1};}}
\title{Functions with a maximal number of finite invariant or internally-1-quasi-invariant sets or supersets}
\author{Nizar El Idrissi and Samir Kabbaj}

\newcommand{\Addresses}{{
  \bigskip
  \footnotesize

  \textbf{Nizar El Idrissi.}
  \par\nopagebreak Laboratoire : Equations aux dérivées partielles, Algèbre et Géométrie spectrales.
  \par\nopagebreak
  Département de mathématiques, faculté des sciences, université Ibn Tofail, 14000 Kénitra.\par\nopagebreak 
  \textit{E-mail address} : \texttt{nizar.elidrissi@uit.ac.ma}

  \medskip

  \textbf{Pr. Samir Kabbaj.} \par\nopagebreak Laboratoire : Equations aux dérivées partielles, Algèbre et Géométrie spectrales.
  \par\nopagebreak
  Département de mathématiques, faculté des sciences, université Ibn Tofail, 14000 Kénitra.\par\nopagebreak 
  \textit{E-mail address} : \texttt{samir.kabbaj@uit.ac.ma}
  
}}

\theoremstyle{plain}

\newtheorem{proposition}{Proposition}[section]
\newtheorem{corollary}{Corollary}[section]
\newtheorem{lemma}{Lemma}[section]

\theoremstyle{definition}
\newtheorem{definition}{Definition}[section]
\newtheorem{example}{Example}[section]   

\theoremstyle{remark}
\newtheorem{remark}{Remark}[section]

\makeatletter
\def\keywords{\xdef\@thefnmark{}\@footnotetext}
\makeatother

\begin{document}
\newpage
\maketitle
\begin{abstract}
A relaxation of the notion of invariant set, known as $k$-quasi-invariant set, has appeared several times in the literature in relation to group dynamics. The results obtained in this context depend on the fact that the dynamic is generated by a group. In our work, we consider the notions of invariant and 1-internally-quasi-invariant sets as applied to an action of a function $f$ on a set $I$. We answer several questions of the following type, where $k \in \{0,1\}$: what are the functions $f$ for which every finite subset of $I$ is internally-$k$-quasi-invariant? More restrictively, if $I = \mathbb{N}$, what are the functions $f$ for which every finite interval of $I$ is internally-$k$-quasi-invariant? Last, what are the functions $f$ for which every finite subset of $I$ admits a finite internally-$k$-quasi-invariant superset? This parallels a similar investigation undertaken by C. E. Praeger in the context of group actions.
\end{abstract}



\tableofcontents

\section{Introduction}

Invariant sets play an essential role in the qualitative study of dynamical systems \cite{KatokHasselblatt,BrinStuck}. In group dynamics, a relaxation of this notion, known as \textit{$k$-quasi-invariant sets}, or \textit{$k$-almost-invariant sets} (where $k \in \mathbb{N}$), appears in the works \cite{Brailovsky, Praeger, BrailovskyPasechenikPraeger, Kharazishvili, KimKim, AlaeiyanYoshiara, AlaeiyanRazzaghmanieshi} and the references therein. Most of the results contained in these articles crucially depend on the fact that $G$ is a group. \\
In our work, we depart from the group setting by extending the notion of a $k$-quasi-invariant set under a group action to the context of an action of a set $A$ on another set $I$ (this simply consists of a function $\rho : A \times I \to I$ with no additional requirement). The lack of bijectivity of the functions $\{\rho(a,\cdot)\}_{a \in A}$ forces us to make the following definitions. If we denote by $|.|$ the cardinality function and $\Lambda^a := \{\rho(a,x) : x \in \Lambda\}$ for $a \in A$ and $\Lambda \subseteq I$, then $\Lambda$ is \textit{externally-$k$-quasi-invariant} under the action $\rho$ if
\[ \forall a \in A : |\Lambda^a \setminus \Lambda| \leq k, \] 
or equivalently
\[ \forall a \in A : \exists P \subseteq I : |P| \leq k \text{ and } \Lambda^a \subseteq \Lambda \cup P, \] 
and is \textit{internally-$k$-quasi-invariant} under the action $\rho$ if
\[ \forall a \in A : \exists P \subseteq I : |P| \leq k \text{ and } (\Lambda \setminus P)^a \subseteq \Lambda. \] 
We focus in this article on internally-$k$-quasi-invariant sets when $A$ is a singleton and $k \in \{0,1\}$. In this case, the action $\rho : A \times I \to I$ becomes just a function $f : I \to I$, and therefore a set $\Lambda$ is internally-$0$-quasi-invariant under the action of $f$ if it is a (forward)-invariant set of $f$, i.e.
\begin{equation}
f(\Lambda) \subseteq \Lambda
\end{equation}
and is internally-$1$-quasi-invariant under the action of $f$ if
\begin{equation}	
\label{eq-1}
\exists a \in \Lambda : f(\Lambda \setminus \{a\}) \subseteq \Lambda.
\end{equation}
Note that the notions of external or internal $k$-quasi-invariance under the action of $f$ that we have just defined, differ from external or internal $k$-quasi-invariance under the action of the monoid $\mathbb{N}$ (by iterates of a function $f$), which we will not be concerned about in this article. \\ 

Most of the findings in \cite{Brailovsky, Praeger, BrailovskyPasechenikPraeger, Kharazishvili, KimKim, AlaeiyanYoshiara, AlaeiyanRazzaghmanieshi} about $k$-quasi-invariant sets under a group action tend to be wrong for internally-$k$-quasi-invariant sets under the action of $f : I \to I$. In \cite{Praeger}, the author shows that if every finite subset of $I$ is $k$-quasi-invariant under a group action with no fixed points, then $I$ must be finite. The analogous statement for internally-1-quasi-invariant sets under the action of a single function $f$ is wrong, as can be attested by the function $succ : \mathbb{N} \to \mathbb{N}$. However, the problem considered by the author in \cite{Praeger} is interesting, and can be adapted to our setting: what are the functions $f : I \to I$ for which every finite subset of $I$ is internally-$k$-quasi-invariant? More restrictively, if $I = \mathbb{N}$, what are the functions $f : I \to I$ for which every finite interval of $I$ is internally-$k$-quasi-invariant? Last, what are the functions $f : I \to I$ for which every finite subset of $I$ admits a finite internally-$k$-quasi-invariant superset? In this article, we investigate these types of questions in the $k \in \{0,1\}$ cases. Besides, we inform the reader that the idea of this article emerged when we were looking for an example to a proposition related to linear independence in (\cite{ElIdrissiKabbaj}, Proposition 6.1 and Example 6.2).  

\section{Notations}
In the sequel, $\mathbb{N}$ denotes the set $\{0,1,2,\cdots\}$ of natural numbers including 0. $\mathbb{N}^*$ denotes $\mathbb{N} \setminus \{0\}$. \\
If $A$ is a set, we denote by $|A|$ the cardinality of $A$, $\mathcal{P}(A)$ the powerset of $A$, $\mathcal{P}_{\omega}(A)$ the set $\{B \subseteq A : |B| < \infty\}$, $\mathcal{P}_{\omega,*}(A)$ the set $\{B \subseteq A : 0 < |B| < \infty\}$, and $\mathcal{P}_{\omega,n^+}(A)$ the set $\{B \subseteq A : n \leq |B| < \infty\}$. \\
If $(A,\leq)$ is a non-empty totally ordered set, we denote by $Int_{\omega,*}(A)$ the set $\{ [a,b] : a \leq b, | [a,b] | < \infty\}$, and by $Int_{\omega, n^+}(A)$ the set $\{ [a,b] : a \leq b, n \leq | [a,b] | < \infty\}$. \\
If $A$ is a set, $\phi : A \to A$ a self map and $n \in \mathbb{N}^*$, we denote by $\phi^n$ the composition of $\phi$ with itself $n$ times : $\phi \circ \cdots \circ \phi : A \to A$. In addition, we define $\phi^0$ to be the identity function on $A$. Moreover, if $a \in A$, we denote by $\mathcal{O}^{+}_\phi(a)$ the forward orbit of $a$ under the iterates of $\phi$ : $\{\phi^n(a) : n \in \mathbb{N}\}$. 

\section{Preliminaries on discrete time forward orbits}

\subsection{Elementary lemmas}

The following well-known lemmas are easy to prove and are only reminded for convenience.

\begin{lemma} 
\label{lemmaDistinct}
Let $I$ be an infinite set, $a \in I$ and $\phi : I \to I$. Then 
\[ \mathcal{O}^{+}_\phi(a) \text { is infinite } \Leftrightarrow a, \phi(a), \phi(\phi(a)), \cdots \text{ are distinct} \]
or equivalently
\[ \mathcal{O}^{+}_\phi(a) \text { is finite } \Leftrightarrow \text{ the sequence } a, \phi(a), \phi(\phi(a)), \cdots \text{ is eventually periodic}. \]
\end{lemma}

\begin{proof}
($\Rightarrow$) Suppose that $\mathcal{O}^{+}_\phi(a)$ is infinite. Suppose that $\phi^{n}(a) = \phi^{m}(a)$ for some $n < m$. By induction, we have $\phi^{n+j}(a) = \phi^{m+j}(a)$ for all $j \in \mathbb{N}$. Let $e \geq n$. Let $e - n = q(m-n) + r$ be the division with remainder of $e - n \in \mathbb{N}$ by $m-n \in \mathbb{N}^*$. If $q \geq 1$, we have $\phi^e(a) = \phi^{n + q(m-n) + r}(a) = \phi^{m + (q-1)(m-n) + r}(a) = \phi^{n + (q-1)(m-n) + r}(a)$. By immediate induction, we have that $\phi^e(a) = \phi^{n+r}(a)$, where $0 \leq r < m-n$. So $\mathcal{O}^{+}_\phi(a) = \{a, \phi(a), \cdots, \phi^{m-1}(a)\}$ is finite, contradiction. \\
($\Leftarrow$) This is clear.
\end{proof}

Therefore, for any map $\phi : I \to I$ and $a \in I$, the following simple conjugation result for $\phi_{|\mathcal{O}^{+}_\phi(a)}$ follows.

\begin{lemma}
Let $I$ be an infinite set, $a \in I$ and $\phi : I \to I$.
\begin{enumerate}
\item Suppose that $\mathcal{O}^{+}_\phi(a)$ is infinite. Then $\phi_{|\mathcal{O}^{+}_\phi(a)}$ is conjugate to $succ : \begin{cases} \mathbb{N} &\to \mathbb{N} \\ n &\mapsto n+1 \end{cases}$.
\item Suppose that $\mathcal{O}^{+}_\phi(a)$ is finite. Then ${\exists (u,v) \in \mathbb{N}\times \mathbb{N}^*}$ and a bijection ${\alpha :  [0,u+v-1] \to \mathcal{O}^{+}_\phi(a)}$ such that 
\[ {\mathcal{O}^{+}_\phi(a) = \{a,\phi(a),\cdots,\phi^{u-1}(a)\} \sqcup \{\phi^u(a), \cdots, \phi^{u+v-1}(a)\}}, \]
\[ \phi_ {| \{a,\phi(a),\cdots,\phi^{u-1}(a)\}} \text{ is conjugated to } succ_{[u]} : \begin{cases} [0,u-1] &\to [1,u] \\ n &\mapsto n+1 \end{cases} \text{ by } \alpha_{|[0,u-1]}, \]
and 
\begin{align*}
&\phi_{| \{\phi^u(a),\cdots,\phi^{u+v-1}(a)\}} \text{ is conjugated to } \\
&\quad cycle_{[v]} :  \begin{cases} [u,u+v-1] &\to [u,u+v-1] \\ n &\mapsto n+1 \text { if } n \in [u,u+v-2] \\ u+v-1 &\mapsto u \end{cases} \text{ by } \alpha_{|[u,u+v-1]}. 
\end{align*}
\end{enumerate}
\end{lemma}

\begin{proof}
The map $\alpha$ that realizes the conjugation is, in both cases, the one that sends $n$ to $\phi^n(a)$.
\end{proof}

Thus, we can deduce

\begin{corollary}
\label{corOrbCofMapBijec}
Let $I$ be an infinite set, $a \in I$ and $\phi : I \to I$. \\
Suppose that $\mathcal{O}^{+}_\phi(a)$ is infinite. \\
Then $\phi_{|\mathcal{O}^{+}_\phi(a)} : \mathcal{O}^{+}_\phi(a) \to \mathcal{O}^{+}_\phi(a) \setminus \{a\}$ is well-defined and bijective.
\end{corollary}

\begin{lemma}
\label{lemmaTwoOrbInfAndFinDon'tIntersect}
Let $I$ be an infinite set, $(a,b) \in I^2$ and $\phi : I \to I$. \\
Suppose that $\mathcal{O}^{+}_\phi(a)$ is infinite and $\mathcal{O}^{+}_\phi(b)$ is finite. \\
Then $\mathcal{O}^{+}_\phi(a)$ and $\mathcal{O}^{+}_\phi(b)$ don't have an intersection point.
\end{lemma}

\begin{proof}
Assume by way of contradiction that $(\exists u,v \in \mathbb{N}) : \phi^u(a) = \phi^v(b)$. Then $\mathcal{O}^{+}_\phi(a) = \{a,\phi(a),\cdots,\phi^{u-1}(a)\} \cup \mathcal{O}^{+}_\phi(\phi^v(b)) \subseteq  \{a,\phi(a),\cdots,\phi^{u-1}(a)\} \cup \mathcal{O}^{+}_\phi(b)$, which is finite, a contradiction.
\end{proof}

\begin{lemma}
\label{lemmaTwoOrbInfAndCofIntersectCofinitely}
Let $I$ be an infinite set, $(a,b) \in I^2$ and $\phi : I \to I$. \\
Suppose that $\mathcal{O}^{+}_\phi(a)$ is cofinite and $\mathcal{O}^{+}_\phi(b)$ is infinite. \\
Then $\mathcal{O}^{+}_\phi(a)$ and $\mathcal{O}^{+}_\phi(b)$ have at least one intersection point.
\end{lemma}

\begin{proof}
Assume by way of contradiction that $\mathcal{O}^{+}_\phi(a)$ and $\mathcal{O}^{+}_\phi(b)$ don't intersect. Then $\mathcal{O}^{+}_\phi(b) \subseteq I \setminus \mathcal{O}^{+}_\phi(a)$. But this is impossible since $\mathcal{O}^{+}_\phi(b)$ is infinite and $I \setminus \mathcal{O}^{+}_\phi(a)$ is finite. Hence the result.
\end{proof}

\begin{corollary}
\label{corollaryACofOrbIntersectsAnOrbitWhenThisOrbitIsInfinite}
Let $I$ be an infinite set, $(a,b) \in I^2$ and $\phi : I \to I$. \\
Suppose that $\mathcal{O}^{+}_\phi(a)$ is cofinite. \\
Then $\mathcal{O}^{+}_\phi(a)$ intersects $\mathcal{O}^{+}_\phi(b)$ if and only if $\mathcal{O}^{+}_\phi(b)$ is infinite.
\end{corollary}

\begin{proof}
Combine lemmas \ref{lemmaTwoOrbInfAndFinDon'tIntersect} and \ref{lemmaTwoOrbInfAndCofIntersectCofinitely}.
\end{proof}

\begin{lemma}
\label{lemmaOrbCofImplications}
Let $I$ be an infinite set, $(a,b) \in I^2$ and $\phi : I \to I$. \\
Suppose that $\mathcal{O}^{+}_\phi(a)$ is cofinite and $\mathcal{O}^{+}_\phi(b)$ is infinite. Then $\mathcal{O}^{+}_\phi(a) \cap \mathcal{O}^{+}_\phi(b)$, and consequently $\mathcal{O}^{+}_\phi(b)$, are cofinite.
\end{lemma}

\begin{proof}
First, notice that by lemma \ref{lemmaDistinct}, $a, \phi(a), \phi^2(a), \cdots$ are distinct. From the hypothesis, there exist $u,v \in \mathbb{N}$ such that $\phi^u(a) = \phi^v(b)$. So 
\begin{align*} 
I \setminus (\mathcal{O}^{+}_\phi(b) \cap \mathcal{O}^{+}_\phi(b)) &= \left( I \setminus  \mathcal{O}^{+}_\phi(a) \right) \cup \left[ \left(I \setminus \mathcal{O}^{+}_\phi(\phi^v(b)) \right) \setminus \{b,\cdots,\phi^{v-1}(b)\} \right] \\
&=  \left( I \setminus  \mathcal{O}^{+}_\phi(a) \right) \cup \left[ \left(I \setminus \mathcal{O}^{+}_\phi(\phi^u(a)) \right) \setminus \{b,\cdots,\phi^{v-1}(b)\} \right] \\
&=  \left( I \setminus  \mathcal{O}^{+}_\phi(a) \right) \\
&\quad \cup \left[ \left( (I \setminus \mathcal{O}^{+}_\phi(a)) \cup \{a,\cdots,\phi^{u-1}(a)\} \right) \setminus \{b,\cdots,\phi^{v-1} (b)\} \right] \\
&\subseteq  (I \setminus  \mathcal{O}^{+}_\phi(a)) \cup \{ a,\cdots,\phi^{u-1}(a) \}
\end{align*}
is finite. The fact that $\mathcal{O}^{+}_\phi(b)$ is also cofinite is due to the inclusion $\mathcal{O}^{+}_\phi(a) \cap \mathcal{O}^{+}_\phi(b) \subseteq \mathcal{O}^{+}_\phi(b)$.
\end{proof}

\begin{corollary}
\label{corollaryOrbPhiIAllOrbCof}
Let $I$ be an infinite set, $a \in I$ and $\phi : I \to I$. \\
Suppose that $\mathcal{O}^{+}_\phi(a) = I$. Then $\forall b \in I : \mathcal{O}^{+}_\phi(b)$ is cofinite
\end{corollary}

\begin{proof}
This can be seen directly or as a consequence of corollary \ref{corollaryACofOrbIntersectsAnOrbitWhenThisOrbitIsInfinite} and lemma \ref{lemmaOrbCofImplications} (since $\forall b \in I : b \in \mathcal{O}^{+}_\phi(a)$).
\end{proof}

\begin{remark}
The converse statement is false. Indeed, let $\phi : \begin{cases} \mathbb{N} \cup \{\bullet\} &\to \mathbb{N} \cup \{\bullet\} \\ n &\mapsto n+1 \\ \bullet &\mapsto 1 \end{cases}$. 
Then $\forall a \in \mathbb{N} \cup \{\bullet\} : (\mathcal{O}^{+}_\phi(a) \text{ is cofinite and } \mathcal{O}^{+}_\phi(a) \neq \mathbb{N} \cup \{\bullet\})$.
\end{remark}

\subsection{Nearest element to a finite set lying in the intersection of orbits of elements of that set}

We now associate to each self-map $\phi : I \to I$ on an infinite set $I$, a particular \textit{noncanonical} map $\xi_\phi : D_\phi \to I$, which gives for each $I^* \in D_\phi$ one of the nearest elements to $I^*$ that lies in the intersection of orbits of elements of $I^*$. The value of $\xi_\phi(I^*)$, for $I^* \in D_\phi$, is uniquely determined if all the orbits of elements in $I^*$ are infinite (use lemmas \ref{lemmaIntersOrb} and \ref{lemmaDistinct}), which arises if and only if $\cap_{a \in I^*} \mathcal{O}^{+}_\phi(a)$ is infinite by corollary \ref{corollaryIntersectionOfFinitelyManyOrbitsIsInfinite}. However, if $\cap_{a \in I^*} \mathcal{O}^{+}_\phi(a)$ is finite, which arises if and only if all the orbits of elements in $I^*$ are finite by corollary \ref{lemmaInclusionOfDPhi}, then $\xi_\phi(I^*)$ may be non-uniquely determined as in the case where $I^*$ is precisely a finite cycle of a map $\phi : I \to I$. This definition will be used in the rest of this article.

\begin{definition}
\label{defXi}
Let $I$ be an infinite set and $\phi : I \to I$. \\
Let $D_\phi = \{I^* \in \mathcal{P}_{\omega,*}(I) : \bigcap_{a \in I^*} \mathcal{O}^{+}_\phi(a) \neq \emptyset \} \subseteq \mathcal{P}_{\omega,*}(I)$. \\
For all $I^* \in D_\phi$, $z \in \bigcap_{a \in I^*} \mathcal{O}^{+}_\phi(a)$ and $a \in I^*$, we define $m^z_a$ as $\min\{m \in \mathbb{N} : \phi^m(a) = z\}$. \\
We define a \textit{noncanonical} map $\xi_\phi : D_\phi \to I$ by selecting for all $I^* \in D_\phi$, $\xi_\phi(I^*) \in \bigcap_{a \in I^*} \mathcal{O}^{+}_\phi(a)$ with $\sum_{a \in I^*} m^{\xi_\phi(I^*)}_a$ minimal.
\end{definition}

\begin{example}
If $I$ is infinite and $I^* \in \mathcal{P}_{\omega,*}(I)$ is such that $(\forall a \in I^*) : \mathcal{O}^{+}_\phi(a)$ is cofinite, then $I^* \in D_\phi$, because a finite intersection of cofinite sets is cofinite and hence non-empty.
\end{example}

\begin{lemma}
\label{lemmaIntersOrb}
Let $I$ be an infinite set, $\phi : I \to I$ and $I^* \in \mathcal{P}_{\omega,*}(I)$ such that ${\bigcap_{a \in I^*} \mathcal{O}^{+}_\phi(a) \neq \emptyset}$. Then $\bigcap_{a \in I^*} \mathcal{O}^{+}_\phi(a) = \mathcal{O}^{+}_\phi(\xi_\phi(I^*))$.
\end{lemma}

\begin{proof}
Suppose that $z \in \bigcap_{a \in I^*} \mathcal{O}^{+}_\phi(a)$. So $(\forall a \in I^*) : z = \phi^{m^z_a}(a)$. Also, we have $(\forall a \in I^*) : \xi_\phi(I^*) = \phi^{m^{\xi_\phi(I^*)}_a}(a)$ where the $m^{\xi_\phi(I^*)}_a$'s satisfy a minimality property. \\
Hence $\sum_{a \in I^*} m^{z}_a \geq \sum_{a \in I^*} m^{\xi_\phi(I^*)}_a$, and so $(\exists a \in I^*) : m^{z}_a \geq m^{\xi_\phi(I^*)}_a$. This implies that 
\[ z = \phi^{m^z_a}(a) = \phi^{m^z_a - m^{\xi_\phi(I^*)}_a}(\phi^{m^{\xi_\phi(I^*)}_a}(a)) = \phi^{m^z_a - m^{\xi_\phi(I^*)}_a}(\xi_\phi(I^*)) \in \mathcal{O}^{+}_\phi(\xi_\phi(I^*)). \] 
Conversely, suppose that $z \in \mathcal{O}^{+}_\phi(\xi_\phi(I^*))$. So $(\exists n \in \mathbb{N}) : z = \phi^n(\xi_\phi(I^*))$. Also, we have $(\forall a \in I^*) : \xi_\phi(I^*) = \phi^{m^{\xi_\phi(I^*)}_a}(a)$. Therefore $(\forall a \in I^*) : z = \phi^{n+m^{I^*}_a}(a) \in \mathcal{O}^{+}_\phi(a)$. \\
Hence the two sets are equal.
\end{proof}

\begin{lemma}
\label{lemmaNonEmptyIntersectionOfTwoInfOrbIsInfinite}
Let $I$ be an infinite set, $(a,b,c) \in I^3$ and $\phi : I \to I$. \\
Suppose that $\mathcal{O}^{+}_\phi(a)$ and $\mathcal{O}^{+}_\phi(b)$ are infinite and have at least one intersection point, and that $c \in \mathcal{O}^{+}_\phi(b)$  \\
Then $\mathcal{O}^{+}_\phi(a) \cap \mathcal{O}^{+}_\phi(c)$ is infinite.
\end{lemma}

\begin{proof}
Since $\mathcal{O}^{+}_\phi(a)$ and $\mathcal{O}^{+}_\phi(b)$ intersect, there exist $(m,n) \in \mathbb{N}^2$ such that $\phi^n(a) = \phi^m(b)$. Also, since $c \in \mathcal{O}^{+}_\phi(b)$, there exists $p \in \mathbb{N}$ such that $c = \phi^p(b)$. Therefore $\phi^{n+p}(a) = \phi^{m+p}(b) = \phi^m(c)$, and so this common element belongs to the orbits of $a$ and $c$.  Moreover, since $\mathcal{O}^{+}_\phi(a)$ is infinite, and this common element and its iterates by $\phi$ belong to it, we have by lemma \ref{lemmaDistinct} that the orbit of this element is infinite. This orbit is included in the orbits of $a$ and $c$, and so in their intersection, which proves that $\mathcal{O}^{+}_\phi(a) \cap \mathcal{O}^{+}_\phi(c)$ is infinite.
\end{proof}

\begin{corollary}
\label{corollaryIntersectionOfFinitelyManyOrbitsIsInfinite}
Let $I$ be an infinite set and $\phi : I \to I$. \\
Let $I^* = \{a_1,\cdots,a_n\} \in \mathcal{P}_{\omega,*}(\{a \in I : |\mathcal{O}^{+}_\phi(a)|=+\infty\})$ such that the orbits $\mathcal{O}^{+}_\phi(a_i)$ intersect jointly or in pairs. 
Then $\cap_{a \in I^*} \mathcal{O}^{+}_\phi(a)$ is infinite.
\end{corollary}

\begin{proof}
If the orbits intersect jointly, then they intersect in pairs. Therefore it suffices to consider the latter case. Now, we have $\forall i \in [\![1,n]\!] : |\mathcal{O}^{+}_\phi(a_i)| = +\infty$, so using lemma \ref{lemmaNonEmptyIntersectionOfTwoInfOrbIsInfinite} (in conjunction with lemma \ref{lemmaIntersOrb}) repeatedly, we see that $| \bigcap_{i \in [\![1,n]\!]} \mathcal{O}^{+}_\phi(a_i) | = \infty$.
\end{proof}

The following lemma asserts that we cannot find a subset $I^* \in D_\phi$ containing both an element of infinite orbit and another element of finite orbit.

\begin{lemma}
\label{lemmaInclusionOfDPhi}
Let $I$ be an infinite set and $\phi : I \to I$. Then we have \\
$D_\phi \subseteq \mathcal{P}_{\omega,*}(\{a \in I : |\mathcal{O}^{+}_\phi(a)|=+\infty\}) \cup \mathcal{P}_{\omega,*}(\{a \in I : |\mathcal{O}^{+}_\phi(a)| < \infty\}) $
\end{lemma}

\begin{proof}
Let $I^* \in D_\phi$. Suppose by way of contradiction that $\{a,b\} \subseteq I^*$ with $\mathcal{O}^{+}_\phi(a) = +\infty$ and $\mathcal{O}^{+}_\phi(b) < +\infty$. Let $c \in \mathcal{O}^{+}_\phi(a) \cap \mathcal{O}^{+}_\phi(b)$. If $\mathcal{O}^{+}_\phi(c)$ were finite, we would have $\phi^n(c) = \phi^m(c)$ for some $m \neq n$ by lemma \ref{lemmaDistinct}, and so $\phi^{n+u}(a) = \phi^{m+u}(a)$ for some $u \in \mathbb{N}$ since $c \in \mathcal{O}^{+}_\phi(a)$, which contradicts lemma \ref{lemmaDistinct} since $\mathcal{O}^{+}_\phi(a)$ is infinite. So $\mathcal{O}^{+}_\phi(c)$ must be infinite. But $c \in \mathcal{O}^{+}_\phi(b)$ implies $\mathcal{O}^{+}_\phi(c) \subseteq \mathcal{O}^{+}_\phi(b)$ which is finite, a contradiction. So either $I^* \in \mathcal{P}_{\omega,*}(\{a \in I : |\mathcal{O}^{+}_\phi(a)|=+\infty\})$ or $I^* \in \mathcal{P}_{\omega,*}(\{a \in I : |\mathcal{O}^{+}_\phi(a)| < \infty\})$.
\end{proof}

\subsection{Maps with pairwise intersecting infinite orbits}

The next definition will be helpful in the next section.

\begin{definition}
Let $I$ be a set and $\phi : I \to I$. We let $\widetilde{P}(\phi)$ be the proposition
\[ \forall a,b \in I : \left[ (|\mathcal{O}^{+}_\phi(a)|=|\mathcal{O}^{+}_\phi(b)|=\infty) \Rightarrow | \mathcal{O}^{+}_\phi(a) \cap \mathcal{O}^{+}_\phi(b) | = \infty \right]. \]
\end{definition}

\begin{remark}
\label{remarkPtilde}
Using lemmas \ref{lemmaNonEmptyIntersectionOfTwoInfOrbIsInfinite} and \ref{lemmaTwoOrbInfAndFinDon'tIntersect}, $\widetilde{P}(\phi)$ is easily seen to be equivalent to
\[ \forall a,b \in I : \left[ (|\mathcal{O}^{+}_\phi(a)|=|\mathcal{O}^{+}_\phi(b)|=\infty) \Rightarrow \mathcal{O}^{+}_\phi(a) \cap \mathcal{O}^{+}_\phi(b) \neq \emptyset \right]. \]
\end{remark}

\begin{example}
\begin{enumerate}
\item $\widetilde{P}(\phi)$ is vacuously true when $I$ is finite.
\item $\widetilde{P}(\phi)$ is true when $(\exists a \in I) : \mathcal{O}^{+}_\phi(a)$ is cofinite, see lemma \ref{lemmaOrbCofPtildeHolds}.
\end{enumerate}
\end{example}

\begin{lemma}
\label{lemmaCharacOfPTildeInTermsOfDPhi}
Let $I$ be an infinite set and $\phi : I \to I$. Then we have \\
$\widetilde{P}(\phi) \Leftrightarrow D_\phi \supseteq \mathcal{P}_{\omega,*}(\{a \in I : |\mathcal{O}^{+}_\phi(a)|=+\infty\})$.
\end{lemma}

\begin{proof}
($\Rightarrow$) Suppose that $\widetilde{P}(\phi)$ is true. Let $I^* = \{a_1,\cdots,a_n\} \in \mathcal{P}_{\omega,*}(\{a \in I : |\mathcal{O}^{+}_\phi(a)|=+\infty\})$. Then it follows from corollary \ref{corollaryIntersectionOfFinitelyManyOrbitsIsInfinite} that $I^* \in D_\phi$. \\
($\Leftarrow$) Suppose that $D_\phi \supseteq \mathcal{P}_{\omega,*}(\{x \in I : |\mathcal{O}^{+}_\phi(x)|=+\infty\})$. Let $a,b \in I$ such that $|\mathcal{O}^{+}_\phi(a)| = |\mathcal{O}^{+}_\phi(b)| = +\infty$. So $\{a,b\} \in \mathcal{P}_{\omega,*}(\{x \in I : |\mathcal{O}^{+}_\phi(x)|=+\infty\}) \subseteq D_\phi$. Therefore $\mathcal{O}^{+}_\phi(a) \cap \mathcal{O}^{+}_\phi(b) \neq \emptyset$ by definition of $D_\phi$. Hence $\widetilde{P}(\phi)$ is true by remark \ref{remarkPtilde}.
\end{proof}

\begin{lemma}
\label{lemmaOrbCofPtildeHolds}
Let $I$ be an infinite set, $\phi : I \to I$, and $a \in I$ such that $\mathcal{O}^{+}_\phi(a)$ is cofinite. Then $\widetilde{P}(\phi)$ holds.
\end{lemma}

\begin{proof}
Let $b_1,b_2 \in I$ such that $|\mathcal{O}^{+}_\phi(b_1)|=|\mathcal{O}^{+}_\phi(b_2)|=\infty$. Then by lemma \ref{lemmaOrbCofImplications}, $\mathcal{O}^{+}_\phi(a) \cap \mathcal{O}^{+}_\phi(b_1)$ is non-empty and cofinite. By lemma \ref{lemmaIntersOrb}, $\mathcal{O}^{+}_\phi(a) \cap \mathcal{O}^{+}_\phi(b_1) = \mathcal{O}^{+}_\phi(\xi_\phi(\{a,b_1\}))$. Again, by lemma \ref{lemmaOrbCofImplications}, $\mathcal{O}^{+}_\phi(\xi_\phi(\{a,b_1\})) \cap \mathcal{O}^{+}_\phi(b_2) = \mathcal{O}^{+}_\phi(a) \cap \mathcal{O}^{+}_\phi(b_1) \cap \mathcal{O}^{+}_\phi(b_2)$ is non-empty and cofinite. Since this set is included in $\mathcal{O}^{+}_\phi(b_1) \cap \mathcal{O}^{+}_\phi(b_2)$, we have that $\mathcal{O}^{+}_\phi(b_1) \cap \mathcal{O}^{+}_\phi(b_2)$ is also non-empty. and cofinite. Therefore $\widetilde{P}(\phi)$ holds.
\end{proof}

\section{Functions with a maximal number of finite invariant sets}
\label{sectionExactPreservUnderDirectImage}

The results of this section are obvious and only proved for completeness. They express the fact that, as long as $k$ is strictly less than the cardinality of $I$, the identity function is the only function with a maximal number of finite invariant sets in $\mathcal{P}_{\omega,k^+}(I)$ or $Int_{\omega, k^+}(\mathbb{N})$.

\subsection{Preservation of all finite sets}

\begin{proposition}
\label{propCharMapPreservingSets}
Let $I$ be a set, $\alpha : I \to I$, and $k$ a natural number such that $1 \leq k \leq |I|$. Suppose that
\[ \forall I^* \in \mathcal{P}_{\omega,k^+}(I) : \alpha(I^*) \subseteq I^*. \]
Then we have \\
(a) if $1 \leq k < |I|$, then $\alpha = Id$. \\
(b) if $k = |I|$, then $\alpha$ can be arbitrary.
\end{proposition}

\begin{proof}
Suppose $1 \leq k < |I|$. If $\alpha(a) \neq a$ for some $a \in I$, one can choose $I^* \in \mathcal{P}_{\omega,k^+}(I)$ such that $a \in I^*$ but $\alpha(a) \notin I^*$ to obtain a contradiction. Conversely the identity function satisfies the requirement. \\
Suppose now that $k = |I|$. Then $I$ is a finite set and $\mathcal{P}_{\omega, k^+}(I) = \{I\}$, and every function $\alpha : I \to I$ obviously preserves $I$.
\end{proof}

\subsection{Preservation of all finite intervals in $\mathbb{N}$}

\begin{proposition}
\label{propCharMapPreservingInt}
Let $\alpha : \mathbb{N} \to \mathbb{N}$, and $k$ a natural number such that $k \geq 1$. Suppose that
\[ \forall [\![a,b]\!] \in Int_{\omega, k^+}(\mathbb{N}) : \alpha([\![a,b]\!]) \subseteq [\![a,b]\!]. \]
Then $\alpha = Id$.
\end{proposition}

\begin{proof}
Let $a \in \mathbb{N}$. If $\alpha(a) > a$, one can choose $[\![a-k+1,a]\!]$ to obtain a contradiction. If $\alpha(a) < a$, one can choose $[\![a,a+k-1]\!]$ to obtain a contradiction. Conversely the identity function satisfies the requirement.
\end{proof}

\section{Functions with a maximal number of finite internally-1-quasi-invariant sets}
\label{sectionPreservUnderDirectImageUpToOneElement}

\subsection{Preservation of all finite sets up to one element}

If $I$ is a set such that $|I| \geq 3$, the next proposition shows that there are two 3-parameter families of functions $\beta : I \to I$ such that there exists a function $w : \mathcal{P}_{\omega,*}(I) \to I$ such that 
\[ \forall I^* \in \mathcal{P}_{\omega,*}(I) : w(I^*) \in I^* \text{ and } \beta(I^* \setminus \{w(I^*) \}) \subseteq I^*. \] Note that we cannot require that $\forall I^* \in \mathcal{P}_{\omega,*}(I) : \beta(w(I^*)) \notin I^*$ since these functions have many fixed points (except if $|I| = 3$).

\begin{proposition}
\label{propBetaAllP*I}
Let $I$ be a set such that $|I| \geq 3$, $\beta : I \to I$ and $w : \mathcal{P}_{\omega,*}(I) \to I$ such that 
\[ \forall I^* \in \mathcal{P}_{\omega,*}(I) : w(I^*) \in I^* \text{ and } \beta(I^* \setminus \{w(I^*)\}) \subseteq I^*. \] 
Then either \\
\textbf{Case \circled{1}.}
\[ \exists (a,b,c) \in I^3 : (\beta(a)=b \text{ and } \beta(b)=c) \text{ and } (\forall x \in I \setminus \{a,b\} : \beta(x) = x) \]
and
\[ w = \begin{cases} \mathcal{P}_{\omega,*}(I) &\to I \\ I^* &\mapsto b \text{ if } \{a,b\} \subseteq I^* \\ I^* &\mapsto a \text{ if } a \in I^* \text{ and } b \notin I^* \\ I^* &\mapsto b \text{ if } a \notin I^* \text{ and } b \in I^* \\ I^* &\mapsto \text{any element of } I^* \text{ otherwise} \end{cases} \]
or \\
\textbf{Case \circled{2}.} 
\[ \exists (a,b,c) \in I^3 \text{ with } a \neq b, b \neq c, c \neq a, \beta(a)=b, \beta(b)=c, \beta(c)=a \]
and
\[ \forall x \in I \setminus \{a,b,c\} : \beta(x) = x \]
and
\[ w = \begin{cases} \mathcal{P}_{\omega,*}(I) &\to I \\ I^* &\mapsto a \text{ if } a \in I^* \text{ and } b,c \notin I^* \\ I^* &\mapsto b \text{ if } b \in I^* \text{ and } a,c \notin I^* \\ I^* &\mapsto c \text{ if } c \in I^* \text{ and } a,b \notin I^* \\ I^* &\mapsto b \text{ if } \{a,b\} \subseteq I^* \text{ and } c \notin I^* \\ I^* &\mapsto a \text{ if } \{a,c\} \subseteq I^* \text{ and } b \notin I^* \\ I^* &\mapsto c \text{ if } \{b,c\} \subseteq I^* \text{ and } a \notin I^* \\ I^* &\mapsto  \text{any element of } I^* \text{ otherwise }
\end{cases}. \]
\end{proposition}

\begin{proof}
If $\beta$ is the identity function, then $\beta$ is clearly of type \circled{1}. \\
Otherwise, there exists $p \in I : \beta(p) \neq p$. \\
If $\beta$ is the identity function on $I \setminus \{p\}$, then $\beta$ is clearly of type \circled{1}. \\
Otherwise, let $q \in I \setminus \{p\}$ such that $\beta(q) \neq q$. \\
If for all $(m,n) \in I \setminus \{p\}$ with $m \neq n$ we had $w(\{p,m\})=m$ and $w(\{p,n\})=n$, then $\beta(p) \in \{p,m\} \cap \{p,n\} = \{p\}$, a contradiction. So for all $(m,n) \in I \setminus \{p\}$ with $m \neq n$, we have $w(\{p,m\})=p$ or $w(\{p,n\}=p$ which implies $\beta(m) \in \{p,m\}$ or $\beta(n) \in \{p,n\}$. Similarly, for all $(m,n) \in I \setminus \{q\}$ with $m \neq n$, $\beta(m) \in \{q,m\}$ or $\beta(n) \in \{q,n\}$. \\
Let $n \in I \setminus \{p,q,\beta(q)\}$. If $w(\{q,n\})=n$, then $\beta(q) \in \{q,n\}$ so $\beta(q)=n$, a contradiction. So $w(\{q,n\})=q$ and therefore $\beta(n) \in \{q,n\}$. \\ 
If $\exists r \in I \setminus \{p,q,\beta(q)\}$ such that $\beta(r) \neq r$, then $\beta(r) = q \notin \{p,r\}$, which implies $\forall m \in I \setminus \{p,r\} : \beta(m) \in \{p,m\}$. In particular $\beta(q) \in \{p,q\}$ and so $\beta(q)=p$. Since $\beta(r) \neq r$, we can prove in the same way as before that $\forall m,n \in I \setminus \{r\}$ with $m \neq n$, $\beta(m) \in \{r,m\}$ or $\beta(n) \in \{r,n\}$. So either $\beta$ is the identity function on $I \setminus \{p,r\}$ or there exists $s \in I \setminus \{p,r\}$ such that $\forall n \in I \setminus \{p,r,s\} : \beta(n) \in \{q,n\} \cap \{r,n\} = \{n\}$ and so $\beta$ is the identity function on $I \setminus \{p,r,s\}$. The first case is included in the second, and we have necessarily $s=q$ since $q \notin \{p,r\}$ and $\beta(q) = p \neq q$. Since $p \neq q$ and $p,q \in I \setminus \{r\}$, we have $\beta(p) \in \{r,p\}$ or $\beta(q) \in \{r,q\}$. The second case is impossible so $\beta(p) \in \{r,p\}$ which implies $\beta(p)=r$. So $\beta$ is of type \circled{2} ($a=p, b=r, c=q$). \\
Otherwise, we have $\forall n \in I \setminus \{p,q,\beta(q)\} : \beta(n) = n$. \\
If $\beta(q)=p$, then $\beta$ is clearly of type \circled{1}. \\
Suppose then that $\beta(q) \neq p$. \\
If $\beta(p) \neq q$, then $\beta(\{p,q\} \setminus \{w(\{p,q\})\}) \subseteq \{p,q\}$ leads to a contradiction independently of $w(\{p,q\} \in \{p,q\}$. So $\beta(p)=q$. \\
Since $q \neq \beta(q)$ and $q,\beta(q) \in I \setminus \{p\}$, we have $\beta(q) \in \{p,q\}$ or $\beta(\beta(q)) \in \{p,\beta(q)\}$. The first case is impossible. So $\beta(\beta(q)) \in \{p,\beta(q)\}$. \\
If $\beta(\beta(q)) = p$ then $\beta$ is of type \circled{2} ($a=p, b=q, c=\beta(q)$). \\
Otherwise $\beta(\beta(q))=\beta(q)$ and so $\beta$ is of type \circled{1} ($a=p, b=q, c=\beta(q)$). \\
The statements about $w$ are easy to prove.
\end{proof}

\subsection{Preservation of all finite intervals up to one element in $\mathbb{N}$}

Apart from the successor function $\beta = succ$ with $w : \begin{cases} Int_{\omega,*}(\mathbb{N}) &\to \mathbb{N} \\ [\![a,b]\!] &\mapsto b \end{cases}$, there exist other functions $\beta : \mathbb{N} \to \mathbb{N}$ such that there exists a function $w : Int_{\omega,*}(\mathbb{N}) \to \mathbb{N}$ such that 
\[ \forall [\![a,b]\!] \in Int_{\omega,*}(\mathbb{N}) : w([\![a,b]\!]) \in [\![a,b]\!] \text{ and } \beta([\![a,b]\!] \setminus \{w([\![a,b]\!])\}) \subseteq [\![a,b]\!],\]
as shown by the next proposition. 

\begin{proposition}
\label{propBetaInt}
Let $\beta : \mathbb{N} \to \mathbb{N}$ such that there exists a function $w : Int_{\omega,*}(\mathbb{N}) \to \mathbb{N}$ such that
\[ \forall [\![a,b]\!] \in Int_{\omega,*}(\mathbb{N}) : w([\![a,b]\!]) \in [\![a,b]\!] \text{ and } \beta([\![a,b]\!] \setminus \{w([\![a,b]\!])\}) \subseteq [\![a,b]\!]. \]
Let $\{b_n\}_{n \in [\![0,l]\!]}$ the sequence of non-fixed points of $\beta$ in strictly increasing order, where $l \in \mathbb{N} \cup \{-1, \infty\}$. We extend $\{b_n\}_{n \in [\![0,l]\!]}$ and $\beta$ by setting $b_{-1} = -1$ and $\beta(-1) = 0$, and if $l < \infty$, we set $b_{l+1} = \infty$. Then we have either \\
\textbf{Case \circled{1}}. $\forall n \in [\![0,l]\!] : b_n < \beta(b_n) \leq b_{n+1}$ and
\[ w = \begin{cases} Int_{\omega,*}(\mathbb{N}) &\to \mathbb{N} \\ [\![a,b]\!] &\mapsto b_n \text { if } \exists! n \in [\![0,l]\!] : b_n \in [\![a,b]\!] \text{ and } \beta(b_n) > b \\ [\![a,b]\!] &\mapsto \text{any element of } [\![a,b]\!] \text{ otherwise} \end{cases}. \]
\textbf{Case \circled{2}}. There exists $n^* \in [\![-1,l-1]\!]$ such that $\forall m \leq n^* : b_m < \beta(b_m) \leq b_{m+1}$, $\beta(b_{n^*+1}) < b_{n^*+1}$, $\forall m \in [\![n^*+2,l]\!] : b_{m-1} \leq \beta(b_m) < b_m$, and
\[ w = \begin{cases} Int_{\omega,*}(\mathbb{N}) &\to \mathbb{N} \\ [\![a,b]\!] &\mapsto b_n \text { if } \exists! n \in [\![0,n^*]\!] : b_n \in [\![a,b]\!] \text{ and } \beta(b_n) > b \\ [\![a,b]\!] &\mapsto b_n \text { if } \exists! n \in [\![n^*+1,l]\!] : b_n \in [\![a,b]\!] \text{ and } \beta(b_n) < a \\ [\![a,b]\!] &\mapsto \text{any element of } [\![a,b]\!] \text{ otherwise} \end{cases}. \]
\textbf{Case \circled{3}}. There exists $n^* \in [\![-1,l-1]\!]$ such that $\forall m \leq n^* : b_m < \beta(b_m) \leq b_{m+1}$, $\beta(b_{n^*+1}) > b_{n^*+2}$, $\forall m \in [\![n^*+2,l]\!] : b_{m-1} \leq \beta(b_m) < b_m$, and 
\[ w = \begin{cases} Int_{\omega,*}(\mathbb{N}) &\to \mathbb{N} \\ [\![a,b]\!] &\mapsto b_n \text { if } \exists! n \in [\![0,n^*]\!] : b_n \in [\![a,b]\!] \text{ and } \beta(b_n) > b \\ [\![a,b]\!] &\mapsto b_n \text { if } \exists! n \in [\![n^*+1,l]\!] : b_n \in [\![a,b]\!] \text{ and } \beta(b_n) < a \\ [\![a,b]\!] &\mapsto \text{any element of } [\![a,b]\!] \text{ otherwise} \end{cases}. \]
\end{proposition}

\begin{proof}
Let's show that 
\begin{equation}
\label{Eq1}
\boxed{\forall n \in [\![0,l]\!] : \left( \beta(b_n) < b_n \Rightarrow \forall m \in [\![n+1,l]\!] : b_{m-1} \leq \beta(b_m) < b_m \right) }
\end{equation}
and
\begin{equation}
\label{Eq2}
\boxed{\forall n \in [\![0,l]\!] : \left( \beta(b_n) > b_{n+1} \Rightarrow \forall m \in [\![n+1,l]\!] : b_{m-1} \leq \beta(b_m) < b_m \right) }.
\end{equation}
Suppose that $\beta(b_n) < b_n$ for a certain $n \in \mathbb{N}$ and $n+1 \leq l$. We have ${\beta([\![b_n,b_{n+1}]\!] \setminus \{w([\![b_n,b_{n+1}]\!])\}) \subseteq [\![b_n,b_{n+1}]\!]}$. If $w([\![b_n,b_{n+1}]\!]) = b_{n+1}$, then $b_n \leq \beta(b_n)$ which is a contradiction. Therefore $w([\![b_n,b_{n+1}]\!]) \neq b_{n+1}$ and so $b_n \leq \beta(b_{n+1}) < b_{n+1}$ since $b_{n+1}$ is not a fixed point. By immediate induction, we have ${\forall m \in [\![n+1,l]\!] : b_{m-1} \leq \beta(b_m) < b_m}$. \\
Therefore assertion \ref{Eq1} holds. \\
Suppose that $\beta(b_n) > b_{n+1}$ for a certain $n \in \mathbb{N}$ and $n + 1 \leq l$. We have ${\beta([\![b_n,b_{n+1}]\!] \setminus \{w([\![b_n,b_{n+1}]\!])\}) \subseteq [\![b_n,b_{n+1}]\!]}$. If $w([\![b_n,b_{n+1}]\!]) = b_{n+1}$, then ${\beta(b_n) \leq b_{n+1}}$ which is a contradiction. Therefore $w([\![b_n,b_{n+1}]\!]) \neq b_{n+1}$ and so ${b_n \leq \beta(b_{n+1}) < b_{n+1}}$ since $b_{n+1}$ is not a fixed point. By immediate induction, we have ${\forall m \in [\![n+1,l]\!] : b_{m-1} \leq \beta(b_m) < b_m}$. \\
Therefore assertion \ref{Eq2} holds. \\
With these assertions in hand, we can now start the proof. We treat the $l = \infty$ and $l < \infty$ cases simultaneously. \\
\textbf{Case \circled{1}}. Suppose $\forall n \in [\![0,l]\!] : b_n < \beta(b_n) \leq b_{n+1}$. \\
Let $a \leq b \in \mathbb{N}$. \\
\textbf{Subcase \circled{1.1}}. Suppose $\exists n \in [\![0,l]\!] : b_n \in [\![a,b]\!] \text{ and } \beta(b_n) > b$. Let's show that there is a unique such $n$. Suppose that $n_1 > n_2$ both satisfy the condition. Since, $n_2 > n_1$, we have $n_2 \geq n_1 + 1$, and since $(b_n)_{n \in \mathbb{N}}$ is strictly increasing, we have 
\[ b \geq b_{n_2} \geq b_{n_1 + 1} \geq \beta(b_{n_1}) > b, \]
which is a contradiction. Therefore there exists a unique $n(a,b) \in [\![0,l]\!]$ such that $b_{n(a,b)} \in [\![a,b]\!] \text{ and } \beta(b_{n(a,b)}) > b$. If we had $w([\![a,b]\!]) \neq b_{n(a,b)}$, then since $\beta([\![a,b]\!]\setminus \{w([\![a,b]\!])\}) \subseteq [\![a,b]\!]$, we would have $\beta(b_{n(a,b)}) \in [\![a,b]\!]$, which contradicts $\beta(b_{n(a,b)}) > b$. Therefore $w([\![a,b]\!]) = b_{n(a,b)}$. \\
\textbf{Subcase \circled{1.2}}. Suppose $\forall n \in [\![0,l]\!] : (b_n \in [\![a,b]\!] \Rightarrow \beta(b_n) \leq b)$. There is no additional constraint in this subcase, we can choose $w([\![a,b]\!])$ arbitrarily from $[\![a,b]\!]$. \\
\textbf{Conversely}. If $l \in \mathbb{N} \cup \{-1, \infty\}$, $\{b_n\}_{n \in [\![0,l]\!]}$ is a strictly increasing sequence, $\beta : \mathbb{N} \to \mathbb{N}$ is a function such that $\mathbb{N} \setminus \{b_n\}_{n \in [\![0,l]\!]}$ are the fixed points of $\beta$, $\forall n \in [\![0,l]\!] : b_n < \beta(b_n) \leq b_{n+1}$, and 
\[ w = \begin{cases} Int_{\omega,*}(\mathbb{N}) &\to \mathbb{N} \\ [\![a,b]\!] &\mapsto b_{n(a,b)} \text { if } \exists n \in [\![0,l]\!] : b_n \in [\![a,b]\!] \text{ and } \beta(b_n) > b \\ [\![a,b]\!] &\mapsto \text{any element of } [\![a,b]\!] \text{ otherwise} \end{cases}, \]
then $\forall a \leq b \in [\![a,b]\!] : \beta([\![a,b]\!] \setminus \{w([\![a,b]\!])\}) \subseteq [\![a,b]\!]$. Indeed, let $a \leq b \in \mathbb{N}$. Because of the uniqueness property of $b_{n(a,b)}$ (when it exists), any $x \in [\![a,b]\!] \setminus \{w([\![a,b]\!])\}$ is either a fixed point (in which case $\beta(x) = x \in [\![a,b]\!]$), or satisfies $x = b_n$ with $\beta(b_n) \leq b$, in which case, since $b_n < \beta(b_n)$, we have $a \leq x =  b_n < \beta(b_n) = \beta(x) \leq b$. \\

\noindent Suppose now $\neg \left( \forall n \in [\![0,l]\!] : b_n < \beta(b_n) \leq b_{n+1} \right)$. Necessarily, $l \geq 0$. Let $n^* = \max \{ n \in [\![-1,l]\!] : \forall m \leq n : b_m < \beta(b_m) \leq b_{m+1} \}$. Necessarily, $n^* \in [\![-1,l-1]\!]$. Notice that if $\beta(b_{n^* +1}) \leq b_{n^*+2}$, then by maximality of $n^*$, we have $\beta(b_{n^*+1}) < b_{n^*+1}$ since $b_{n^*+1}$ is not a fixed point of $\beta$. Therefore we have either $\beta(b_{n^*+1}) < b_{n^*+1}$ or $\beta(b_{n^*+1}) > b_{n^*+2}$. \\

\noindent \textbf{Case \circled{2}}.  Suppose $\forall m \leq n^* : b_m < \beta(b_m) \leq b_{m+1}$ and $\beta(b_{n^*+1}) < b_{n^*+1}$. By assertion \ref{Eq1}, we have $\forall m \in [\![n^*+2,l]\!] : b_{m-1} \leq \beta(b_m) < b_m$. \\
Let $a \leq b \in \mathbb{N}$. \\
\textbf{Subcase \circled{2.1}}. Suppose $\exists n \in [\![0,n^*]\!] : b_n \in [\![a,b]\!] \text{ and } \beta(b_n) > b$. Let's show that there is a unique such $n$. Suppose that $n_1 > n_2$ both satisfy the condition. Since, $n_2 > n_1$, we have $n_2 \geq n_1 + 1$, and since $(b_n)_{n \in \mathbb{N}}$ is strictly increasing, we have 
\[ b \geq b_{n_2} \geq b_{n_1 + 1} \geq \beta(b_{n_1}) > b, \]
which is a contradiction. Therefore there exists a unique $n(a,b) \in [\![0,n^*]\!]$ such that $b_{n(a,b)} \in [\![a,b]\!] \text{ and } \beta(b_{n(a,b)}) > b$. If we had $w([\![a,b]\!]) \neq b_{n(a,b)}$, then since $\beta([\![a,b]\!]\setminus \{w([\![a,b]\!])\}) \subseteq [\![a,b]\!]$, we would have $\beta(b_{n(a,b)}) \in [\![a,b]\!]$, which contradicts $\beta(b_{n(a,b)}) > b$. Therefore $w([\![a,b]\!]) = b_{n(a,b)}$. \\
\textbf{Subcase \circled{2.2}}. Suppose $\exists n \in [\![n^*+1,l]\!] : b_n \in [\![a,b]\!] \text{ and } \beta(b_n) < a $. Let's show that there is a unique such $n$. Suppose that $n_1 > n_2$ both satisfy the condition. Since, $n_2 > n_1$, we have $n_2 -  1 \geq n_1$ and $n_2 \geq n^*+2$, and since $(b_n)_{n \in \mathbb{N}}$ is strictly increasing, we have 
\[ a > \beta(b_{n_2}) \geq b_{n_2 - 1} \geq b_{n_1} \geq a, \]
which is a contradiction. Therefore there exists a unique $n'(a,b) \in [\![n^*+1,l]\!]$ such that $b_{n'(a,b)} \in [\![a,b]\!] \text{ and } \beta(b_{n'(a,b)}) > b$. If we had $w([\![a,b]\!]) \neq b_{n'(a,b)}$, then since $\beta([\![a,b]\!]\setminus \{w([\![a,b]\!])\}) \subseteq [\![a,b]\!]$, we would have $\beta(b_{n'(a,b)}) \in [\![a,b]\!]$, which contradicts $\beta(b_{n'(a,b)}) > b$. Therefore $w([\![a,b]\!]) = b_{n'(a,b)}$. \\
Let's show that subcases 2.1 and 2.1 are disjoint. Suppose the contrary. Then 
\begin{align*}
a > \beta(b_{n'(a,b)}) &\geq b_{n'(a,b)-1}  > \beta(b_{n'(a,b)-1}) \geq b_{n'(a,b)-2} > \cdots  \\
&> \beta(b_{n^*+2}) \geq b_{n^* + 1} \geq \beta(b_{n^*}) > b_{n^*} \geq \beta(b_{n^* - 1}) > \cdots > \beta(b_{n(a,b)}) > b, 
\end{align*}
which is a contradiction. \\
\textbf{Subcase \circled{2.3}}. Suppose $\forall n \in [\![0,n^*]\!] : (b_n \in [\![a,b]\!] \Rightarrow \beta(b_n) \leq b)$ and ${\forall n \in [\![n^*+1,l]\!] : (b_n \in [\![a,b]\!] \Rightarrow \beta(b_n) \geq a)}$. There is no additional constraint in this subcase, we can choose $w([\![a,b]\!])$ arbitrarily from $[\![a,b]\!]$. \\
\textbf{Conversely}. If $l \geq 0$, $\{b_n\}_{n \in [\![0,l]\!]}$ is a strictly increasing sequence, $\beta : \mathbb{N} \to \mathbb{N}$ is a function such that $\mathbb{N} \setminus \{b_n\}_{n \in [\![0,l]\!]}$ are the fixed points of $\beta$, $n^* \in [\![-1,l-1]\!]$, $\forall m \leq n^* : b_m < \beta(b_m) \leq b_{m+1}$, $\beta(b_{n^*+1}) < b_{n^*+1}$, $\forall m \in [\![n^*+2,l]\!] : b_{m-1} \leq \beta(b_m) < b_m$, and 
\[ w = \begin{cases} Int_{\omega,*}(\mathbb{N}) &\to \mathbb{N} \\ [\![a,b]\!] &\mapsto b_{n(a,b)} \text { if } \exists n \in [\![0,n^*]\!] : b_n \in [\![a,b]\!] \text{ and } \beta(b_n) > b \\ [\![a,b]\!] &\mapsto b_{n'(a,b)} \text { if } \exists n \in [\![n^*+1,l]\!] : b_n \in [\![a,b]\!] \text{ and } \beta(b_n) < a \\ [\![a,b]\!] &\mapsto \text{any element of } [\![a,b]\!] \text{ otherwise} \end{cases}, \]
then $\forall a \leq b \in [\![a,b]\!] : \beta([\![a,b]\!] \setminus \{w([\![a,b]\!])\}) \subseteq [\![a,b]\!]$. Indeed, let $a \leq b \in \mathbb{N}$. Because of the uniqueness properties of $b_{n(a,b)}$ and $b_{n'(a,b)}$ (when they exist) and the disjointness of subcases 2.1 and 2.2, any $x \in [\![a,b]\!] \setminus \{w([\![a,b]\!])\}$ is either a fixed point (in which case $\beta(x) = x \in [\![a,b]\!]$), or satisfies $x = b_n$ with $n \in [\![0,n^*]\!]$ and $\beta(b_n) \leq b$, in which case, since $b_n < \beta(b_n)$, we have $a \leq x =  b_n < \beta(b_n) = \beta(x) \leq b$, or satisfies $x = b_n$ with $n \geq n^* + 1$ and $\beta(b_n) \geq a$, in which case, since $\beta(b_n) < b_n$, we have $a \leq \beta(b_n) = \beta(x) < b_n = x \leq b$.  \\
\textbf{Case \circled{3}}.  Suppose $\forall m \leq n^* : b_m < \beta(b_m) \leq b_{m+1}$ and $\beta(b_{n^*+1}) > b_{n^*+2}$. By assertion \ref{Eq2}, we have $\forall m \geq n^* + 2 : b_{m-1} \leq \beta(b_m) < b_m$. Then the treatment from now on is analogous to case 2 (since we haven't used the conditions $\beta(b_{n^*+1}) < b_{n^*+1}$ or $\beta(b_{n^*+1}) > b_{n^*+2}$) with the same necessary and sufficient form of the function $w$.\\
\end{proof}

Requiring that for such functions, $\beta(w([\![a,b]\!])) \notin [\![a,b]\!]$ for all $a \leq b \in \mathbb{N}$, leads to

\begin{corollary}
Let $\beta : \mathbb{N} \to \mathbb{N}$ and $w : Int_{\omega,*}(\mathbb{N}) \to \mathbb{N}$ be two functions such that 
\begin{align*}
\forall [\![a,b]\!] \in Int_{\omega,*}(\mathbb{N}) &: \quad \quad \quad \quad \quad w([\![a,b]\!]) \in [\![a,b]\!] \\
 &\text{ and } \beta([\![a,b]\!]\setminus \{w([\![a,b]\!])\}) \subseteq [\![a,b]\!] \text{ and } \beta(w([\![a,b]\!])) \notin [\![a,b]\!].
\end{align*}
Then either \\
\textbf{Case \circled{1}}. $\beta : \begin{cases} \mathbb{N} &\to \mathbb{N} \\ n &\mapsto n+1 \end{cases}$ and $w : \begin{cases} Int_{\omega,*}(\mathbb{N}) &\to \mathbb{N} \\ [\![a,b]\!] &\mapsto b \end{cases}$ \\
\textbf{Case \circled{2}}. There exists $n^* \neq u \in \mathbb{N}$ such that $\beta : \begin{cases} \mathbb{N} &\to \mathbb{N} \\ n &\mapsto n+1 \text { for } n < n^* \\ n^* &\mapsto u \\ n &\mapsto n-1 \text{ for } n > n^* \end{cases}$ and $w : \begin{cases} Int_{\omega,*}(\mathbb{N}) &\to \mathbb{N} \\ [\![a,b]\!] &\mapsto b \text{ if } b \in [\![0,n^*]\!] \\  [\![a,b]\!] &\mapsto a \text{ if } a \geq n^*+  1 \end{cases}$.
\end{corollary}

\begin{proof}
Notice that for all $a \in \mathbb{N}$, $\beta(w([a,a])) \notin [a,a]$ implies that $\beta(a)\neq a$, and so the sequence of non-fixed points of $\beta$ in strictly increasing order is $\{b_n\}_{n \in \mathbb{N}} = \{n\}_{n \in \mathbb{N}}$. We can then apply proposition \ref{propBetaInt}. \\
\end{proof}

\section{Functions with a maximal number of finite invariant supersets}
\label{sectionExactSuperPreservUnderDirectImage}

\subsection{Solution to the general problem} 

\begin{proposition}
\label{propCharMapPreservingSupersets}
Let $\alpha : I \to I$ and $G : \mathcal{P}_{\omega,*}(I) \to \mathcal{P}_{\omega,*}(I)$. Then the following conditions are equivalent
\begin{enumerate}
\item $\forall I^* \in \mathcal{P}_{\omega,*}(I) : I^* \subseteq G(I^*) \text{ and } \alpha(G(I^*)) \subseteq G(I^*)$. 
\item The orbits of $\alpha$ are finite and there exists $H : \mathcal{P}_{\omega,*}(I) \to \mathcal{P}_{\omega}(I)$ such that 
\[ \forall I^* \in \mathcal{P}_{\omega,*}(I) : G(I^*) = \bigcup_{a \in I^* \cup H(I^*)} \mathcal{O}^{+}_\alpha(a). \]
\end{enumerate}
\end{proposition}

\begin{proof}
(1) $\Rightarrow$ (2). Let $a \in I$. We have $\alpha(a) \in G(\{a\})$ since $a \in G(\{a\})$ and $\alpha(G(\{a\})) \subseteq G(\{a\})$. By immediate induction, we have $\mathcal{O}^{+}_\alpha(a) \subseteq G(\{a\})$. Since $G(\{a\})$ is finite, this implies that $\mathcal{O}^{+}_\alpha(a)$ is finite. Therefore the orbits of $\alpha$ are finite. Furthermore, let $H(I^*) = G(I^*)$ for all $I^* \in \mathcal{P}_{\omega,*}(I)$. Let $I^* \in \mathcal{P}_{\omega,*}(I)$. We have $I^* \cup H(I^*) = I^* \cup G(I^*) = G(I^*)$. Let $a \in G(I^*)$. We have $\alpha(a) \in G(I^*)$ since $a \in G(I^*)$ and $\alpha(G(I^*)) \subseteq G(I^*)$. By immediate induction, we have $\mathcal{O}^{+}_\alpha(a) \subseteq G(I^*)$. Therefore $\bigcup_{a \in G(I^*)} \mathcal{O}^{+}_\alpha(a) \subseteq G(I^*)$. The reverse inclusion is obvious since $a \in \mathcal{O}^{+}_\alpha(a)$ for all $a \in G(I^*)$. \\
(2) $\Rightarrow$ (1). Let $I^* \in \mathcal{P}_{\omega,*}(I)$. To show that $I^* \subseteq G(I^*)$, let $a \in I^*$. We have ${a \in \mathcal{O}^{+}_\alpha(a) \subseteq G(I^*)}$. Therefore $I^* \subseteq G(I^*)$. To show that $\alpha(G(I^*)) \subseteq G(I^*)$, let $a \in G(I^*)$. We have \[
\alpha(a) \in \alpha \left(\bigcup_{a \in I^* \cup H(I^*)} \mathcal{O}^{+}_\alpha(a) \right) = \bigcup_{a \in I^* \cup H(I^*)} \alpha \left( \mathcal{O}^{+}_\alpha(a) \right) \subseteq \bigcup_{a \in I^* \cup H(I^*)} \mathcal{O}^{+}_\alpha(a). \]
Therefore $\alpha(G(I^*)) \subseteq G(I^*)$.
\end{proof}

\subsection{Solution to a more constrained problem in $\mathbb{N}$}

The following proposition gives a characterization of the functions $\alpha : \mathbb{N} \to \mathbb{N}$ and $G : \mathcal{P}_{\omega,*}(\mathbb{N}) \to \mathcal{P}_{\omega,*}(\mathbb{N})$ that satisfy the conditions of proposition \ref{propCharMapPreservingSupersets}, plus the additional hypothesis that $ \forall I^* \in \mathcal{P}_{\omega,*}(\mathbb{N}) : \max(G(I^*)) \in \alpha(I^*)$.

\begin{proposition}
\label{propCharMapPreservingSupersetsWithMaxCondition}
Let $\alpha : \mathbb{N} \to \mathbb{N}$ and $G : \mathcal{P}_{\omega,*}(\mathbb{N}) \to \mathcal{P}_{\omega,*}(\mathbb{N})$ such that 
\[ \forall I^* \in \mathcal{P}_{\omega,*}(\mathbb{N}) : I^* \subseteq G(I^*) \text{, } \alpha(G(I^*)) \subseteq G(I^*) \text{ and } \max(G(I^*)) \in \alpha(I^*). \]
Then $\alpha$ admits an infinite number of fixed points, and letting this sequence be $\{b_n\}_{n \in \mathbb{N}}$ in strictly increasing order, there exists a function $j : \mathbb{N} \to \mathbb{N}$ such that 
\[ \forall a \in \mathbb{N} : \alpha(a) = b_{j(a)} \]
and
\begin{align*}
\forall a \in \mathbb{N} : j(a) &= n \text{ if } \exists n \in \mathbb{N} : a = b_n \\
j(a) &\geq r(a) \text{ if } a \notin \{b_n\}_{n \in \mathbb{N}}
\end{align*}
where $r(a) := \min\{n \in \mathbb{N} : b_n > a \}$, and there exists a function $H : \mathcal{P}_{\omega,*}(\mathbb{N}) \to  \mathcal{P}_{\omega}(\mathbb{N})$ such that $\forall I^* \in  \mathcal{P}_{\omega,*}(\mathbb{N}) : G(I^*) = \bigcup_{a \in I^* \cup H(I^*)} \{a,\alpha(a)\}$ and $\forall I^* \in  \mathcal{P}_{\omega,*}(\mathbb{N}) : \exists a \in I^* : \forall x \in H(I^*) : j(x) \leq j(a)$. Conversely, any such $\alpha$ and $G$ are valid solutions. \\
\end{proposition}

\begin{proof}
\textbf{Direct implication.} First, let's show that 
\begin{equation}
\label{eq0}
\boxed{\forall n \in \mathbb{N} : \alpha(\alpha(n)) = \alpha(n) \geq n}
\end{equation}
We have $\forall n \in \mathbb{N} : \max(G(\{n\})) \in \alpha(\{n\})$, so $\alpha(n) = \max(G(\{n\})) \in G(\{n\})$. Let $n \in \mathbb{N}$. We have $\alpha(\alpha(n)) \in G(\{n\}$ since $\alpha(n) \in G(\{n\})$ (we use the hypothesis $\alpha(G(I^*)) \subseteq G(I^*)$ for $I^* = \{n\}$), and so $\alpha(\alpha(n)) \leq \alpha(n)$ since $\alpha(n) = \max(G(\{n\}))$. Moreover, we have $n \in \{n\} \subseteq G(\{n\})$, and so $n \leq \max(G(\{n\})) = \alpha(n)$. Replacing $n$ by $\alpha(n)$, we have $\forall n \in \mathbb{N} : \alpha(\alpha(n)) \geq \alpha(n)$. \\
Therefore $\forall n \in \mathbb{N} : \alpha(\alpha(n)) = \alpha(n) \geq n$, and so assertion \ref{eq0} holds. \\
Notice that if $\alpha$ is injective, then by assertion \ref{eq0} we have $\forall n \in \mathbb{N} : \alpha(n) = n$, and $\alpha$ is the identity function. Otherwise, $\alpha$ admits a sequence of fixed points $\{\alpha(n)\}_{n \in \mathbb{N}}$ tending to infinity since $\forall n \in \mathbb{N} : \alpha(n) \geq n$. Let $\{b_n\}_{n \in \mathbb{N}}$ be the sequence of all fixed points of $\alpha$ in strictly increasing order. Let $b_{-1} = -1$ and extend $\alpha$ by setting $\alpha(-1)=-1$. Then $\forall n \in \mathbb{N} \cup \{-1\} : \forall m \in [b_n + 1, b_{n+1} - 1] : \exists k \geq n+1 : \alpha(m) = b_k$ since $\alpha(m)$ is a fixed point and $\alpha(m) \geq m > b_n$. Moreover, we have $r(m) \leq n+1$ since $b_{n+1} > m$, and so $r(m) \leq k$. Therefore we can set $j(m) = n$ if $m = b_n$ and $j(m)=k$ if $m \in [b_n+1,b_{n+1}-1]$ for some $n \in \mathbb{N} \cup \{-1\}$. \\
Moreover, from proposition \ref{propCharMapPreservingSupersets}, there exists a function $H : \mathcal{P}_{\omega,*}(\mathbb{N}) \to  \mathcal{P}_{\omega}(\mathbb{N})$ such that $\forall I^* \in  \mathcal{P}_{\omega,*}(\mathbb{N}) : G(I^*) = \bigcup_{a \in I^* \cup H(I^*)} \mathcal{O}^{+}_\alpha(a) = \bigcup_{a \in I^* \cup H(I^*)} \{a,\alpha(a)\}$. It's easy to show that the condition $\forall I^* \in  \mathcal{P}_{\omega,*}(\mathbb{N}) : \max(G(I^*)) \in \alpha(I^*)$ means that $\forall I^* \in  \mathcal{P}_{\omega,*}(\mathbb{N}) : \exists a \in I^* : \forall x \in H(I^*) : j(x) \leq j(a)$. \\ \\
\textbf{Conversely.} Given a function $\alpha : \mathbb{N} \to \mathbb{N}$ which admits an infinite number of fixed points $\{b_n\}_{n \in \mathbb{N}}$ in strictly increasing order, such that there exists a function $j : \mathbb{N} \to \mathbb{N}$ such that 
\[ \forall a \in \mathbb{N} : \alpha(a) = b_{j(a)} \]
and
\begin{align*}
\forall a \in \mathbb{N} : j(a) &= n \text{ if } \exists n \in \mathbb{N} : a = b_n \\
j(a) &\geq r(a) \text{ if } a \notin \{b_n\}_{n \in \mathbb{N}}
\end{align*}
where $r(a) := \min\{n \in \mathbb{N} : b_n > a \}$, and such that there exists a function $H : \mathcal{P}_{\omega,*}(\mathbb{N}) \to  \mathcal{P}_{\omega}(\mathbb{N})$ such that $\forall I^* \in  \mathcal{P}_{\omega,*}(\mathbb{N}) : G(I^*) = \bigcup_{a \in I^* \cup H(I^*)} \{a,\alpha(a)\}$ and $\forall I^* \in  \mathcal{P}_{\omega,*}(\mathbb{N}) : \exists a \in I^* : \forall x \in H(I^*) : j(x) \leq j(a)$, then it's easy to check that $\forall I^* \in  \mathcal{P}_{\omega,*}(\mathbb{N}) : \alpha(G(I^*)) \subseteq G(I^*)$ and $\max(G(I^*)) \in \alpha(I^*)$.
\end{proof}

\begin{remark}
If $\alpha$ is assumed to be bijective, then we can also use proposition \ref{propIrreducibilityPhiGu} to deduce $\alpha = Id$. Indeed, suppose that $\alpha : \mathbb{N} \to \mathbb{N}$ and $G : \mathcal{P}_{\omega,*}(\mathbb{N}) \to \mathcal{P}_{\omega,*}(\mathbb{N})$ are such that
\[ \forall I^* \in \mathcal{P}_{\omega,*}(\mathbb{N}) : I^* \subseteq G(I^*) \text{, } \alpha(G(I^*)) \subseteq G(I^*) \text{ and } \max(G(I^*)) \in \alpha(I^*), \] 
and let $\beta : \begin{cases}  \mathbb{N} &\to \mathbb{N} \\ n &\mapsto n+1 \end{cases}$, $\phi = \beta \circ \alpha$, and for all $I^* \in \mathcal{P}_{\omega,*}(\mathbb{N})$, $u(I^*) \in I^*$ such that $\alpha(u(I^*)) = \max(G(I^*))$. \\
Then since $P_2(\phi,G,u)$ is true and $\forall I^* \in \mathcal{P}_{\omega,*}(\mathbb{N}) : \exists a \in G(I^*) : |\mathcal{O}^{+}_\beta(a)| = +\infty$, we have $\alpha = Id$.
\end{remark}

In the following proposition, we require furthermore that $G : \mathcal{P}_{\omega,*}(\mathbb{N}) \to Int_{\omega,*}(\mathbb{N})$.

\begin{proposition}
\label{propCharMapPreservingSupersetsWithMaxCondition2}
Let $\alpha : \mathbb{N} \to \mathbb{N}$ and $G : \mathcal{P}_{\omega,*}(\mathbb{N}) \to Int_{\omega,*}(\mathbb{N})$ such that 
\[ \forall I^* \in \mathcal{P}_{\omega,*}(\mathbb{N}) : I^* \subseteq G(I^*) \text{, } \alpha(G(I^*)) \subseteq G(I^*) \text{ and } \max(G(I^*)) \in \alpha(I^*). \]
Then $\alpha$ admits an infinite number of fixed points, and if this sequence is $\{b_n\}_{n \in \mathbb{N}}$ in strictly increasing order, there exists a function $j : \mathbb{N} \to \mathbb{N}$ such that 
\[ \forall a \in \mathbb{N} : \alpha(a) = b_{j(a)}, \]
\begin{align*}
\forall a \in \mathbb{N} : j(a) &= n \text{ if } \exists n \in \mathbb{N} : a = b_{n} \\
j(a) &= r(a)  \text{ if } a > b_0 \text{ and } a \notin \{b_n\}_{n \in \mathbb{N}}
\end{align*}
where $r(a) := \min\{n \in \mathbb{N} : b_n > a \}$, and 
\[ j(0) \geq j(1) \cdots \geq j(b_0 - 1), \]
and there exists a function $u : \mathcal{P}_{\omega,*}(\mathbb{N}) \to \mathbb{N}$ such that $\forall I^* \in  \mathcal{P}_{\omega,*}(\mathbb{N}) : G(I^*) = [\![u(I^*),\max\{\alpha(a)\}_{a \in I^*}]\!]$, and
\[ \forall I^* \in  \mathcal{P}_{\omega,*}(\mathbb{N}) : u^*(I^*) := \min \{n \leq b_0 : \max\{\alpha(a)\}_{a \in I^*} \geq \alpha(n) \} \leq u(I^*) \leq \min(I^*), \]
Conversely, any such $\alpha$ and $G$ are valid solutions.
\end{proposition}

\begin{proof}
\textbf{Direct implication.} $\alpha$ and $G$ satisfy of course all the previous properties of proposition \ref{propCharMapPreservingSupersetsWithMaxCondition}. \\
For all $I^* \in \mathcal{P}_{\omega,*}(\mathbb{N})$, we let $[\![u(I^*),v(I^*)]\!] := G(I^*)$. Assume by way of contradiction that $v(I^*) \neq \max\{\alpha(a)\}_{a \in I^*}$. We have $\forall a \in I^* \subseteq G(I^*) : \alpha(a) \leq v(I^*)$, so $\max\{\alpha(a)\}_{a \in I^*} \leq v(I^*)$. Since $v(I^*) \neq \max\{\alpha(a)\}_{a \in I^*}$, we have $\max\{\alpha(a)\}_{a \in I^*} < v(I^*)$ which is a contradiction since $v(I^*) = \max(G(I^*)) \in \alpha(I^*)$. So $v(I^*) = \max\{\alpha(a)\}_{a \in I^*}$. \\
Applying this to $I^* = \{b_n,b_{n+1}\}$ for $n \in \mathbb{N}$ gives $G(\{b_n,b_{n+1}\}) = [u(\{b_n,b_{n+1}\}),b_{n+1}]$ where $u(\{b_n,b_{n+1}\}) \leq b_n$, and so $\forall x \in [\![b_n,b_{n+1}]\!] : \alpha(x) \leq b_{n+1}$ since $G(I^*)$ is preserved by $\alpha$. Hence $\forall x \in ]\!]b_n,b_{n+1}]\!] : \alpha(x) = b_{n+1}$ (since $j(x) \geq n+1$ by proposition \ref{propCharMapPreservingSupersetsWithMaxCondition}). \\
Applying it for $I^* = \{x\}$ for $x \in [0,b_0[$ gives $G(I^*) = [\![u(\{x\}),\alpha(x)]\!]$ where $u(\{x\}) \leq x$. Since $\alpha(x)$ is a fixed point, we have $\alpha(x) \geq b_0$. Therefore, $\forall y \in [a,b_0[ : \alpha(y) \leq \alpha(x)$ since $G(I^*)$ is preserved by $\alpha$. So $\alpha(0) \geq \alpha(1) \cdots \geq \alpha(b_0-1)$, which is equivalent to $j(0) \geq j(1) \geq \cdots \geq j(b_0 - 1)$.  \\
Moreover, for all $I^* \in \mathcal{P}_{\omega,*}(\mathbb{N})$, we have $u(I^*) \leq \min(I^*)$ since $I^* \subseteq G(I^*)$. \\
Furthermore, for all $I^* \in \mathcal{P}_{\omega,*}(\mathbb{N})$, the set $\{n \leq b_0 : \max\{\alpha(a)\}_{a \in I^*} \geq \alpha(n) \}$ is non-empty (it contains $b_0$), and letting $u^*(I^*) := \min \{n \leq b_0 : \max\{\alpha(a)\}_{a \in I^*} \geq \alpha(n) \}$, we have necessarily $u(I^*) \geq u^*(I^*)$. Indeed, if $u(I^*) > b_0$, the inequality is obvious. Otherwise, it is true since $\alpha(u(I^*)) \in G(I^*)$ (since $\alpha(G(I^*)) \subseteq G(I^*)$ and $u(I^*) \in G(I^*)$), and therefore $\alpha(u(I^*)) \leq \max(G(I^*)) =  \max\{\alpha(a)\}_{a \in I^*}$. \\ \\
\textbf{Conversely.} Given a function $\alpha : \mathbb{N} \to \mathbb{N}$ which admits an infinite number of fixed points $\{b_n\}_{n \in \mathbb{N}}$ in strictly increasing order, such that there exists a function $j : \mathbb{N} \to \mathbb{N}$ such that 
\[ \forall a \in \mathbb{N} : \alpha(a) = b_{j(a)} \]
and
\begin{align*}
\forall a \in \mathbb{N} : j(a) &= n \text{ if } \exists n \in \mathbb{N} : a = b_n \\
j(a) &= r(a)  \text{ if } a > b_0 \text{ and } a \notin \{b_n\}_{n \in \mathbb{N}}
\end{align*}
where $r(a) := \min\{n \in \mathbb{N} : b_n > a \}$, and 
\[ j(0) \geq j(1) \cdots \geq j(b_0 - 1), \]
and such that $\forall I^* \in  \mathcal{P}_{\omega,*}(\mathbb{N}) : G(I^*) = [\![u(I^*),\max\{\alpha(a)\}_{a \in I^*}]\!]$, where $u : \mathcal{P}_{\omega,*}(\mathbb{N}) \to \mathbb{N}$ is a function such that 
\[ \forall I^* \in  \mathcal{P}_{\omega,*}(\mathbb{N}) : u^*(I^*) := \min \{n \leq b_0 : \max\{\alpha(a)\}_{a \in I^*} \geq \alpha(n) \} \leq u(I^*) \leq \min(I^*), \]
then it's clear that $\forall I^* \in  \mathcal{P}_{\omega,*}(\mathbb{N}) : \max(G(I^*)) \in \alpha(I^*)$. To check that $\forall I^* \in  \mathcal{P}_{\omega,*}(\mathbb{N}) : \alpha(G(I^*)) \subseteq G(I^*)$, we take $I^* \in  \mathcal{P}_{\omega,*}(\mathbb{N})$ and $x \in G(I^*)$, and distinguish between the cases $x \geq b_0$ and $x < b_0$. In the first case, since $x \leq \max\{\alpha(a)\}_{a \in I^*}$, $\max\{\alpha(a)\}_{a \in I^*}$ is a fixed point, and $x \geq b_0$, then it follows that $u(I^*) \leq x \leq \alpha(x) \leq \max\{\alpha(a)\}_{a \in I^*}$. In the second case, we have $u^*(I^*) \leq u(I^*) \leq x < b_0$, therefore $\alpha(x) \leq \alpha(u^*(I^*)) \leq \max\{\alpha(a)\}_{a \in I^*}$ since $\alpha$ is non-increasing on $[\![0,b_0-1]\!]$ and by the definition of $u^*(I^*)$. Moreover, since $\alpha(x)$ is a fixed point, it is obvious that $\alpha(x) \geq b_0 > x \geq u(I^*)$. Therefore we are done in both cases and it follows that $\alpha(G(I^*)) \subseteq G(I^*)$.
\end{proof}

\section{Functions with a maximal number of finite internally-1-quasi-invariant supersets}
\label{sectionSuperPreservUnderDirectImageUpToOneElement}

Below, we define two similar predicates (in the sense of propositions depending on mathematical objects) that model the property of direct image preservation by $\phi : I \to I$ of finite supersets of finite subsets of $I$, up to one element, denoted by $P_1(\phi,G,u)$ and $P_2(\phi,G,u)$. We will be concerned with these predicates in the following three subsections.

\begin{definition}
We say that a 3-tuple ($\phi : I \to I, G : \mathcal{P}_{\omega,*}(I) \to \mathcal{P}_{\omega,*}(I), u : \mathcal{P}_{\omega,*}(I) \to I)$ satisfies property $P_1(\phi,G,u)$ if 
\[ \forall I^* \in \mathcal{P}_{\omega,*}(I) : u(I^*) \in G(I^*), I^* \subseteq G(I^*) \text{ and } \phi(G(I^*) \setminus \{u(I^*)\}) \subseteq G(I^*). \]
\end{definition}

\begin{definition}
We say that a 3-tuple ($\phi : I \to I, G : \mathcal{P}_{\omega,*}(I) \to \mathcal{P}_{\omega,*}(I), u : \mathcal{P}_{\omega,*}(I) \to I)$ satisfies property $P_2(\phi,G,u)$ if 
\[ \forall I^* \in \mathcal{P}_{\omega,*}(I) : u(I^*) \in I^*, I^* \subseteq G(I^*) \text{ and } \phi(G(I^*) \setminus \{u(I^*)\}) \subseteq G(I^*). \]
\end{definition}

The following simple lemma is often used (without mention) in the paper.

\begin{lemma}
\label{lemmaRelationsPhiUG}
Let $I$ be an infinite set and ($\phi : I \to I,G : \mathcal{P}_{\omega,*}(I) \to \mathcal{P}_{\omega,*}(I),u : \mathcal{P}_{\omega,*}(I) \to I$) a 3-tuple satisfying $P_1(\phi,G,u)$. Then we have 
\[ \forall I^* \in \mathcal{P}_{\omega,*}(I) : \left( |\mathcal{O}^{+}_\phi(u(I^*))| = \infty \Rightarrow \phi(u(I^*)) \notin G(I^*) \text{ and } \phi(u(I^*)) \notin I^* \right) . \]
\end{lemma}

\begin{proof}
Let $I^* \in \mathcal{P}_{\omega,*}(I)$ such that $|\mathcal{O}^{+}_\phi(u(I^*))| = \infty$. From lemma \ref{lemmaDistinct}, $u(I^*),\phi(u(I^*),\phi^2(u(I^*)),\cdots$ must be distinct. In particular, $\forall n \geq 1, \phi^n(u(I^*)) \neq u(I^*)$. Thus we have $\phi(u(I^*)) \notin G(I^*)$ (otherwise we would have $\mathcal{O}^{+}_\phi(u(I^*)) \subseteq G(I^*)$ which is impossible since $G(I^*)$ is finite and $|\mathcal{O}^{+}_\phi(u(I^*))| = \infty$). Since $I^* \subseteq G(I^*)$, this implies that $\phi(u(I^*)) \notin I^*$.
\end{proof}

We also have

\begin{lemma}
Let $I$ be an infinite set and ($\phi : I \to I,G : \mathcal{P}_{\omega,*}(I) \to \mathcal{P}_{\omega,*}(I),u : \mathcal{P}_{\omega,*}(I) \to I$) a 3-tuple satisfying $P_2(\phi,G,u)$. Then we have
\[ \left( \forall I^* \in \mathcal{P}_{\omega,*}(I) : |\mathcal{O}^{+}_\phi(u(I^*))| = \infty \right) \Rightarrow \left( \forall b \in I : \forall m \in \mathbb{N} : u(\{b,\cdots,\phi^m(b)\} = \phi^m(b) \right). \]
\end{lemma}

\begin{proof}
Suppose that $\forall I^* \in \mathcal{P}_{\omega,*}(I) : |\mathcal{O}^{+}_\phi(u(I^*))| = \infty$. Let  $b \in I$ and $m \in \mathbb{N}$. Since we should have $u(\{b,\cdots,\phi^m(b)\}) \in \{b,\cdots,\phi^m(b)\}$ and $\phi(u(\{b,\cdots, \phi^m(b)\})) \notin \{b,\cdots, \phi^m(b)\}$ by lemma \ref{lemmaRelationsPhiUG}, we have necessarily $u(\{b,\cdots, \phi^m(b)\}) = \phi^m(b)$.
\end{proof}

\subsection{Solution to problem $P_1(\phi,G,u)$}

We first prove the following structural proposition.

\begin{proposition}
\label{propP1}
Let $I$ be an infinite set and $\beta : I \to I$. \\
Let $G : \mathcal{P}_{\omega,*}(I) \to \mathcal{P}_{\omega,*}(I)$ and $v : \mathcal{P}_{\omega,*}(I) \to I$. Then \\
$P_1(\beta,G,v) \Leftrightarrow$ there exist three functions $H,\overline{H}, \widetilde{H} : \mathcal{P}_{\omega,*}(I) \to \mathcal{P}_{\omega}(I)$ such that for all $I^* \in \mathcal{P}_{\omega,*}(I)$,  $I^* \subseteq H(I^*) \cup \overline{H}(I^*) \cup \widetilde{H}(I^*), H(I^*) \subseteq \{a \in I : |\mathcal{O}^{+}_\beta(a)| = +\infty\}$, $\overline{H}(I^*), \widetilde{H}(I^*) \subseteq \{a \in I : |\mathcal{O}^{+}_\beta(a)| < +\infty\}$, and \\
(a) $\exists z \in G(I^*) : |\mathcal{O}^{+}_\beta(z)| = +\infty$ and \\
\[ \begin{cases} H(I^*) &\neq \emptyset \\ \overline{H}(I^*) &= \emptyset \\ G(I^*) &= \left( \bigcup_{a \in \widetilde{H}(I^*)} \mathcal{O}^{+}_\beta(a) \right) \cup \left( \bigcup_{a \in H(I^*)} \{a, \beta(a), \beta^2(a), \cdots, \xi_\beta(H(I^*))\} \right) \\ v(I^*) &= \xi_\beta(H(I^*)) \end{cases} \] 
or \\
(b) $\forall a \in G(I^*) : |\mathcal{O}^{+}_\beta(a)|<+\infty$ and
\[ \begin{cases} H(I^*) &= \emptyset \\ \overline{H}(I^*) &\neq \emptyset \\ G(I^*) &= \left( \bigcup_{a \in \widetilde{H}(I^*)} \mathcal{O}^{+}_\beta(a) \right) \cup \left( \bigcup_{a \in \overline{H}(I^*)} \{a, \beta(a), \beta^2(a), \cdots, v(I^*)\} \right) \end{cases}. \] 
\end{proposition}

\begin{proof}
($\Rightarrow$) : Suppose $P_1(\beta,G,v)$. \\
Let
\[ H : \begin{cases} \mathcal{P}_{\omega,*}(I) &\to \mathcal{P}_{\omega}(I) \\ I^* &\mapsto \{a \in G(I^*) : |\mathcal{O}^{+}_\beta(a)| = +\infty \} \end{cases}, \]
\[ \overline{H} : \begin{cases} \mathcal{P}_{\omega,*}(I) &\to \mathcal{P}_{\omega}(I) \\ I^* &\mapsto \{a \in G(I^*) : |\mathcal{O}^{+}_\beta(a)| < +\infty \text{ and } v(I^*) \in \mathcal{O}^{+}_\beta(a) \} \end{cases} \]
and
\[ \widetilde{H} : \begin{cases} \mathcal{P}_{\omega,*}(I) &\to \mathcal{P}_{\omega}(I) \\ I^* &\mapsto \{a \in G(I^*) : |\mathcal{O}^{+}_\beta(a)| < +\infty \text{ and } v(I^*) \notin \mathcal{O}^{+}_\beta(a) \} \end{cases}. \]
Let $I^* \in \mathcal{P}_{\omega,*}(I)$.
We have $G(I^*) = H(I^*) \cup \overline{H}(I^*) \cup \widetilde{H}(I^*)$. Therefore $I^* \subseteq H(I^*) \cup \overline{H}(I^*) \cup \widetilde{H}(I^*)$. \\
(a) Suppose that $\exists z \in G(I^*) : |\mathcal{O}^{+}_\beta(z)| = +\infty$. So $H(I^*) \neq \emptyset$. \\
Let's show that $v(I^*) = \xi_\beta(H(I^*))$. Assume by way of contradiction that $\exists b \in H(I^*) : v(I^*) \notin \mathcal{O}^{+}_\beta(b)$. Then we must have $\mathcal{O}^{+}_\beta(b) \subseteq G(I^*)$, which is impossible since $|\mathcal{O}^{+}_\beta(b)|=+\infty$ and $G(I^*)$ is finite. So $v(I^*) \in \bigcap_{b \in H(I^*)} \mathcal{O}^{+}_\beta(b) = \mathcal{O}^{+}_\beta(\xi_\beta(H(I^*)))$. In particular, $|\mathcal{O}^{+}_\beta(v(I^*))|=+\infty$ and so $v(I^*) \in H(I^*)$. So there exist $m,n \in \mathbb{N}$ such that $v(I^*) = \beta^m(\xi_\beta(H(I^*)))$ and $\xi_\beta(H(I^*)) = \beta^n(v(I^*))$. This can only work if $m=n=0$ due to lemma \ref{lemmaDistinct} and therefore $v(I^*) = \xi_\beta(H(I^*))$. \\
Now we must have $\overline{H}(I^*) = \emptyset$ (otherwise there exists $a' \in G(I^*)$ such that $|\mathcal{O}^{+}_\beta(a')| < +\infty$ and $v(I^*) = \xi_\beta(H(I^*)) \in \mathcal{O}^{+}_\beta(z) \cap \mathcal{O}^{+}_\beta(a') = \emptyset$ by lemma \ref{lemmaInclusionOfDPhi}, a contradiction). \\
Let's show that $G(I^*) = \left( \bigcup_{a \in \widetilde{H}(I^*)} \mathcal{O}^{+}_\beta(a) \right) \cup \left( \bigcup_{a \in H(I^*)} \{a, \beta(a), \beta^2(a), \cdots, \xi_\beta(H(I^*))\} \right)$. $P_1(\beta,G,v)$ shows that the inclusion $\supseteq$ is true. Conversely, let $a \in G(I^*)$.  We have either $a \in H(I^*)$ or $a \in \widetilde{H}(I^*)$. Therefore the inclusion $\subseteq$ is also true. \\
(b) Suppose that $\forall a \in G(I^*) : |\mathcal{O}^{+}_\beta(a)| < +\infty$. So $H(I^*) = \emptyset$ and $\overline{H}(I^*) \neq \emptyset$. \\
Let's show that $G(I^*) = \left( \bigcup_{a \in \widetilde{H}(I^*)} \mathcal{O}^{+}_\beta(a) \right) \cup \left( \bigcup_{a \in \overline{H}(I^*)} \{a, \beta(a), \beta^2(a), \cdots, v(I^*)\} \right)$. \\
$P_1(\beta,G,v)$ shows that the inclusion $\supseteq$ is true. Conversely, let $a \in G(I^*)$. We have $a \in \widetilde{H}(I^*)$ or $a \in \overline{H}(I^*)$ and so the inclusion $\subseteq$ is true. \\
($\Leftarrow$) : Suppose that $H, \overline{H}$ and $\widetilde{H}$ exist with the desired properties. Let $I^* \in \mathcal{P}_{\omega,*}(I)$. \\
(a) If $\exists z \in G(I^*) : |\mathcal{O}^{+}_\beta(z)| = +\infty$, then $\overline{H}(I^*) = \emptyset$ and so $G(I^*) \supseteq \widetilde{H}(I^*) \cup H(I^*) \supseteq I^*$ using the hypotheses. Moreover, we have clearly $v(I^*) \in G(I^*)$ and $\beta(G(I^*)\setminus \{v(I^*)\}) \subseteq G(I^*)$. \\
(b) If $\forall a \in G(I^*) : |\mathcal{O}^{+}_\beta(a)| < +\infty$, we have $H(I^*) = \emptyset$ and we can check similarly that $P_1(\beta,G,v)$ is true.
\end{proof}

The following proposition characterizes the existential problem $\exists G \exists v : P_1(\beta,G,v)$ to $\widetilde{P}(\beta)$ in terms of $\widetilde{P}(\beta)$.

\begin{proposition}
\label{propCharExistGvP1EquivPtilde}
Let $I$ be an infinite set and $\beta : I \to I$. Then
\[ \left( \exists G : \mathcal{P}_{\omega,*}(I) \to \mathcal{P}_{\omega,*}(I) : \exists v : \mathcal{P}_{\omega,*}(I) \to I : P_1(\beta,G,v) \right)  \Leftrightarrow \widetilde{P}(\beta). \]
\end{proposition}

\begin{proof}
($\Rightarrow$) : Let $a,b \in I$ such that $|\mathcal{O}^{+}_\beta(a)| = |\mathcal{O}^{+}_\beta(b)| = +\infty$. Assume by way of contradiction that $\mathcal{O}^{+}_\beta(a) \cap \mathcal{O}^{+}_\beta(b) = \emptyset$. So $a \neq b$. Let $I^* = \{a,b\}$. Let $G : \mathcal{P}_{\omega,*}(I) \to \mathcal{P}_{\omega,*}(I)$ and $v : \mathcal{P}_{\omega,*}(I) \to I$ such that $P_1(\beta,G,v)$. Suppose by way of contradiction that $v(I^*) \in \mathcal{O}^{+}_\beta(a)$. Then $v(I^*) \notin \mathcal{O}^{+}_\beta(b)$ since $\mathcal{O}^{+}_\beta(a) \cap \mathcal{O}^{+}_\beta(b) = \emptyset$. Using the hypothesis, we obtain $\mathcal{O}^{+}_\beta(b) \subseteq G(I^*)$ which is impossible since $G(I^*)$ is finite and $\mathcal{O}^{+}_\beta(b)$ is infinite. Therefore $v(I^*) \notin \mathcal{O}^{+}_\beta(a)$. But this is also impossible for the same reason. So we must have $\mathcal{O}^{+}_\beta(a) \cap \mathcal{O}^{+}_\beta(b) \neq \emptyset$. \\
($\Leftarrow$) : Suppose that $\widetilde{P}(\beta)$ is true. Let
\[ H : \begin{cases} \mathcal{P}_{\omega,*}(I) &\to \mathcal{P}_{\omega}(I)  \\ I^* &\mapsto \{a \in I^* : |\mathcal{O}^{+}_\beta(a)| = +\infty \} \end{cases}, \]
\[ \widetilde{H} : \begin{cases} \mathcal{P}_{\omega,*}(I) &\to \mathcal{P}_{\omega}(I)  \\ I^* &\mapsto \{a \in I^* : |\mathcal{O}^{+}_\beta(a)| < +\infty \} \end{cases}, \]
\[ G : \begin{cases} \mathcal{P}_{\omega,*}(I) &\to \mathcal{P}_{\omega,*}(I) \\ I^* &\mapsto \left( \bigcup_{a \in \widetilde{H}(I^*)} \mathcal{O}^{+}_\beta(a) \right) \cup \left( \bigcup_{a \in H(I^*)} \{a, \beta(a), \beta^2(a), \cdots, \xi_\beta(H(I^*))\} \right) \end{cases}, \]
and 
\[ v : \begin{cases} \mathcal{P}_{\omega,*}(I) &\to I \\ I^* &\mapsto \xi_\beta(H(I^*)) \text{ if } H(I^*) \neq \emptyset \\ I^* &\mapsto \text{an arbitrary element of } I^* \text{ otherwise} \end{cases}, \]
where we used the fact $\mathcal{P}_{\omega,*}(\{a \in I : |\mathcal{O}^{+}_\beta(a)|=+\infty\}) \subseteq D_\beta$ (see lemma \ref{lemmaCharacOfPTildeInTermsOfDPhi}). Then $P_1(\beta,G,v)$ holds.
\end{proof}

\subsection{Solution to problem $P_2(\phi,G,u)$}

Proposition \ref{propP1} implies, trivially, the following characterization of $P_2(\phi,G,u)$. It is not known to the authors if $P_2(\phi,G,u)$ admits a simpler description.

\begin{proposition}
\label{propP2}
Let $I$ be an infinite set and $\beta : I \to I$. \\
Let $G : \mathcal{P}_{\omega,*}(I) \to \mathcal{P}_{\omega,*}(I)$ and $v : \mathcal{P}_{\omega,*}(I) \to I$. Then \\
$P_2(\beta,G,v) \Leftrightarrow$ there exist three functions $H,\overline{H}, \widetilde{H} : \mathcal{P}_{\omega,*}(I) \to \mathcal{P}_{\omega}(I)$ such that for all $I^* \in \mathcal{P}_{\omega,*}(I)$,  $I^* \subseteq H(I^*) \cup \overline{H}(I^*) \cup \widetilde{H}(I^*), H(I^*) \subseteq \{a \in I : |\mathcal{O}^{+}_\beta(a)| = +\infty\}$, $\overline{H}(I^*), \widetilde{H}(I^*) \subseteq \{a \in I : |\mathcal{O}^{+}_\beta(a)| < +\infty\}$, and \\
(a) $\exists z \in G(I^*) : |\mathcal{O}^{+}_\beta(z)| = +\infty$ and \\
\[ \begin{cases} H(I^*) &\neq \emptyset \\ \overline{H}(I^*) &= \emptyset \\ G(I^*) &= \left( \bigcup_{a \in \widetilde{H}(I^*)} \mathcal{O}^{+}_\beta(a) \right) \cup \left( \bigcup_{a \in H(I^*)} \{a, \beta(a), \beta^2(a), \cdots, \xi_\beta(H(I^*))\} \right) \\ v(I^*) &= \xi_\beta(H(I^*)) \in I^* \end{cases} \] 
or \\
(b) $\forall a \in G(I^*) : |\mathcal{O}^{+}_\beta(a)|<+\infty$ and
\[ \begin{cases} H(I^*) &= \emptyset \\ \overline{H}(I^*) &\neq \emptyset \\ G(I^*) &= \left( \bigcup_{a \in \widetilde{H}(I^*)} \mathcal{O}^{+}_\beta(a) \right) \cup \left( \bigcup_{a \in \overline{H}(I^*)} \{a, \beta(a), \beta^2(a), \cdots, v(I^*)\} \right) \\ v(I^*) &\in I^* \end{cases}. \] 
\end{proposition}

In what follows, we direct our attention to the existential problem $\exists G \exists u : P_2(\phi,G,u)$, which we try to characterize in terms of a condition on $\phi$. To achieve that, let's first prove

\begin{proposition}
\label{propDbetaxibetainIImpliesExistGvP2}
Let $I$ be an infinite set and $\beta : I \to I$. Then, if $D_\beta = \mathcal{P}_{\omega,*}(I)$ and $\forall I^* \in \mathcal{P}_{\omega,*}(I) : \xi_\beta(I^*) \in I^*$, then $\left( \exists G : \mathcal{P}_{\omega,*}(I) \to \mathcal{P}_{\omega,*}(I) : \exists v : \mathcal{P}_{\omega,*}(I) \to I : P_2(\beta,G,v) \right)$.
\end{proposition}

\begin{proof}
Suppose that $D_\beta = \mathcal{P}_{\omega,*}(I)$ and $\forall I^* \in \mathcal{P}_{\omega,*}(I) : \xi_\beta(I^*) \in I^*$. Then, by lemma \ref{lemmaInclusionOfDPhi}, and using the previous notations of proposition \ref{propCharExistGvP1EquivPtilde}, we have $\forall I^* \in \mathcal{P}_{\omega,*}(I) : H(I^*) \in \{I^*,\emptyset\}$ and so $v(I^*) \in I^*$. Therefore, there exist $G$ and $v$ with the desired properties.
\end{proof}

\begin{definition}
Let $I$ be a set and $\phi : I \to I$. \\
We define the partial order $\preceq_\phi$ on $I$ by : $a \preceq_\phi b \Leftrightarrow b \in \mathcal{O}^{+}_\phi(a)$.
\end{definition}

The following proposition characterizes, under the condition that all the orbits of $\phi$ are infinite, the existential problem $\exists G \exists u : P_2(\phi,G,u)$ in terms of the orbital structure of $\phi$.

\begin{proposition}
\label{propOrbitsInfiniteGuExistEquiv}
Let $I$ be an infinite set, $\phi : I \to I$ such that $(\forall a \in I) : \mathcal{O}^{+}_\phi(a)$ is infinite. \\
Then the following conditions are equivalent :
\begin{enumerate}
\item $\exists G : \mathcal{P}_{\omega,*}(I) \to \mathcal{P}_{\omega,*}(I) : \exists u : \mathcal{P}_{\omega,*}(I) \to I : P_2(\phi,G,u)$.
\item $\forall I^* \in \mathcal{P}_{\omega,*}(I) : \left( \bigcap_{a \in I^*} \mathcal{O}^{+}_\phi(a) \right) \cap I^* \neq \emptyset$,
\item $D_\phi = \mathcal{P}_{\omega,*}(I)$ and $\forall I^* \in \mathcal{P}_{\omega,*}(I) : \left( \bigcap_{a \in I^*} \mathcal{O}^{+}_\phi(a) \right) \cap I^* = \{\xi_\phi(I^*)\}$.
\item $D_\phi = \mathcal{P}_{\omega,*}(I)$ and $\forall I^* \in \mathcal{P}_{\omega,*}(I) : \xi_\phi(I^*) \in I^*$,
\item $\forall n \in \mathbb{N}^* : \forall (a_1,\cdots,a_n) \in I^n$ which are distinct : $\exists d \in [\![1,n]\!] : \exists (m_1,\cdots,m_n) \in \mathbb{N}^n$ with $m_d = 0$ : $\forall i \in [\![1,n]\!] : a_d = \phi^{m_i}(a_i)$.
\item $\forall (a,b) \in I^2$ with $a \neq b$, $(\exists n \in \mathbb{N} : a = \phi^n(b))$ or $(\exists n \in \mathbb{N} : b = \phi^n(a))$.
\item $\preceq_\phi$ is a total order on $I$.
\end{enumerate}
\end{proposition}

\begin{proof}
$(1) \Rightarrow (2) :$ Let $I^* \in \mathcal{P}_{\omega,*}(I)$. We have from proposition \ref{propP1} that ${u(I^*) = \xi_\phi(H(I^*)) \in \cap_{a \in H(I^*)} \mathcal{O}^{+}_\phi(a) = \cap_{a \in G(I^*)} \mathcal{O}^{+}_\phi(a) \subseteq \cap_{a \in I^*} \mathcal{O}^{+}_\phi(a)}$, where we used the fact that here we have necessarily $H(I^*) = G(I^*) \supseteq I^*$. Moreover, we have $u(I^*) \in I^*$. So ${\left( \bigcap_{a \in I^*} \mathcal{O}^{+}_\phi(a) \right) \cap I^* \neq \emptyset}$. \\
$(2) \Rightarrow (3) :$ Let $I^* \in \mathcal{P}_{\omega,*}(I)$. Since $\bigcap_{a \in I^*} \mathcal{O}^{+}_\phi(a) \neq \emptyset$, then $I^* \in D_\phi$. Let $x \in \left( \bigcap_{a \in I^*} \mathcal{O}^{+}_\phi(a) \right) \cap I^*$.  Since $x \in \bigcap_{a \in I^*} \mathcal{O}^{+}_\phi(a) = \mathcal{O}^{+}_\phi(\xi_\phi(I^*))$ by lemma \ref{lemmaIntersOrb} and $\xi_\phi(I^*) = \phi^{m^{\xi_\phi(I^*)}_x}(x)$ because $x \in I^*$ (see definition \ref{defXi}), then $\exists n \geq m^{\xi_\phi(I^*)}_x : x = \phi^n(x)$. Since $\{x, \phi(x), \phi^2(x), \cdots\}$ are distinct (see lemma \ref{lemmaDistinct}), we have necessarily $n=0$ and so $m^{\xi_\phi(I^*)}_x = 0$. This implies $\xi_\phi(I^*) = \phi^{m^{\xi_\phi(I^*)}_x}(x) = x \in I^*$. Since this is true for all $x \in \left( \bigcap_{a \in I^*} \mathcal{O}^{+}_\phi(a) \right) \cap I^*$, we have $\left( \bigcap_{a \in I^*} \mathcal{O}^{+}_\phi(a) \right) \cap I^* = \{\xi_\phi(I^*)\}$. \\
$(3) \Rightarrow (4) :$ This is clear. \\
$(4) \Rightarrow (1) :$ This is proposition \ref{propDbetaxibetainIImpliesExistGvP2}. \\
$(4) \Rightarrow (5) :$ Let $n \in \mathbb{N}^*$ and let $a_1,\cdots,a_n$ be distinct elements of $I$. Applying (4) to $I^* = \{a_1,\cdots,a_n\}$ gives $\xi_\phi(I^*) \in I^*$ which is the content of (5). \\
$(5) \Rightarrow (2) :$ Let $I^* \in \mathcal{P}_{\omega,*}(I)$. Write $I^*$ as $\{a_1,\cdots,a_n\}$ where $n$ is the cardinality of $I^*$. From (5), we have $a_d \in \left( \bigcap_{a \in I^*} \mathcal{O}^{+}_\phi(a) \right) \cap I^*$, and so this last set is not empty. \\
$(5) \Rightarrow (6) :$ This is clear. \\
$(6) \Rightarrow (5) :$ Let $n \in \mathbb{N}^*$ and let $a_1,\cdots,a_n$ be distinct elements of $I$. Assume by way of contradiction that $\forall d \in [\![1,n]\!] : \exists l(d) \in [\![1,n]\!] \setminus \{d\} : a_d \notin \mathcal{O}^{+}_\phi(a_{l(d)})$. By $(6)$, this implies that $\forall d \in [\![1,n]\!] : a_{l(d)} \in \mathcal{O}^{+}_\phi(a_d)$. Besides, we have necessarily $\exists d \in [\![1,n]\!] : \exists e \in \mathbb{N}^* : l^e(d) = d$. Otherwise, we would have $\forall d \in [\![1,n]\!] : \forall e \in \mathbb{N}^* : l^e(d) \neq d$, which implies that $\forall \widetilde{d} \in [\![1,n]\!] : \{\widetilde{d}, l(\widetilde{d}), l^2(\widetilde{d}), \cdots\}$ are distinct. This is impossible since this infinite set is included in $[\![1,n]\!]$. Therefore, $\exists (n_1,\cdots,n_e) \in \mathbb{N}^e : a_d = a_{l^e(d)} = \phi^{n_e}(a_{l^{e-1}(d)}) = \phi^{n_e}(\phi^{n_{e-1}}(a_{l^{e-2}(d)})) = \cdots = \phi^{\sum_{i=1}^e n_i}(a_d)$. Since $\mathcal{O}^{+}_\phi(a_d)$ is infinite, we have by lemma \ref{lemmaDistinct} that $\forall i \in [1,e] : n_i = 0$, which implies in particular that $a_{l(d)} = a_d$, a contradiction. Therefore we have the result.  \\
$(6) \Leftrightarrow (7)$ : This is just the definition of a total order.
\end{proof}

From this proposition, the following corollary characterizing the existential problem $\exists G \exists u : P_2(\phi,G,u)$ with no additionnal assumption can be easily derived.

 \begin{corollary}
\label{corollaryGuExistEquiv}
Let $I$ be an infinite set and $\phi : I \to I$. Denote by $I_\phi^{inf}$ (resp. $I_\phi^{fin}$) the subsets of $I$ consisting of the elements which have an infinite (resp. finite) orbit under $\phi$. Notice that $I_\phi^{inf}$ and $I_\phi^{fin}$ are both invariant under $\phi$ and that $\{ I_\phi^{inf}, I_\phi^{fin} \}$ is a partition of $I$. Then the following conditions are equivalent :
\begin{enumerate}
\item $\exists G : \mathcal{P}_{\omega,*}(I) \to \mathcal{P}_{\omega,*}(I) : \exists u : \mathcal{P}_{\omega,*}(I) \to I : P_2(\phi,G,u)$.
\item $(\preceq_\phi)_{|I_\phi^{inf}}$ is a total order on $I_\phi^{inf}$.
\end{enumerate}
\end{corollary}

\begin{lemma}
\label{lemmaOrbPhiITotalOrder}
Let $I$ be a set and $\phi : I \to I$. \\
Suppose that $\exists a \in I : \mathcal{O}^{+}_\phi(a) = I$. \\
Then $\preceq_\phi$ is a total order on $I$.
\end{lemma}

\begin{proof}
Let $x,y \in I$. We can write $x = \phi^n(a)$ and $y = \phi^m(a)$ for $m,n \in \mathbb{N}$. If $m \leq n$, we have $y \preceq_\phi x$. Otherwise we have $x \preceq_\phi y$. Hence $\preceq_\phi$ is a total order.
\end{proof}

The following proposition is a supplement to proposition \ref{propOrbitsInfiniteGuExistEquiv} where we also suppose the existence of a cofinite orbit.

\begin{proposition}
\label{propTripleCharacOrbPhiI}
Let $I$ be an infinite set and $\phi : I \to I$ such that $\forall a \in I : \mathcal{O}^{+}_\phi(a)$ is infinite. Then we have
\begin{align*}
&( \exists \widetilde{a} \in I : \mathcal{O}^{+}_\phi(\widetilde{a}) \text{ is cofinite} ) \\
& \text{ and } \left( \exists u : \mathcal{P}_{\omega,*}(I) \to I : \exists G : \mathcal{P}_{\omega,*}(I) \to \mathcal{P}_{\omega,*}(I) : P_2(\phi,G,u) \right) \\
& \Leftrightarrow \left( \exists \widetilde{a} \in I : \mathcal{O}^{+}_\phi(\widetilde{a}) \text{ is cofinite} \right) \text{ and } \left( \preceq_\phi \text{ is a total order on } I \right) \\
& \Leftrightarrow \exists \alpha : (\mathbb{N},\leq) \to (I,\preceq_\phi) \text{ such that } \alpha \text{ is an increasing bijection} \\
& \Leftrightarrow \exists \widetilde{a} \in I : \mathcal{O}^{+}_\phi(\widetilde{a}) = I.
\end{align*}
\end{proposition}

\begin{proof}
Suppose that $\exists \widetilde{a} \in I : \mathcal{O}^{+}_\phi(\widetilde{a}) = I$. Let $\alpha = \begin{cases} \mathbb{N} &\to I \\ n &\mapsto \phi^n(\widetilde{a}) \end{cases}$. Clearly, $\alpha$ is increasing. It is injective by lemma \ref{lemmaDistinct}. It is surjective by the hypothesis. \\
Conversely, suppose that there is such an increasing bijection $\alpha$ from $\mathbb{N}$ to $I$. For all $n \in \mathbb{N}$, we have $\alpha(n+1) = \phi(\alpha(n))$. Indeed, since $\alpha$ is increasing, we have $\exists l(n) \in \mathbb{N} : \alpha(n+1) = \phi^{l(n)}(\alpha(n))$. Suppose by way of contradiction that $l(n) \neq 1$. If $l(n)=0$, we would have $\alpha(n+1)=\alpha(n)$, contradicting the injectivity of $\alpha$. If $l(n) \geq 2$, let $m \in \mathbb{N}$ such that $\alpha(m)=\phi(\alpha(n))$. If $m \geq n+1$, then $\alpha(n+1) \preceq_\phi \alpha(m)$, and so $\phi(\alpha(n)) = \alpha(m) \in \mathcal{O}^{+}_\phi(\alpha(n+1)) = \mathcal{O}^{+}_\phi(\phi^{l(n)}(\alpha(n)))$. This implies $\exists u \in \mathbb{N} : \phi(\alpha(n))  = \phi^{l(n)-1+u}(\phi(\alpha(n)))$ where $l(n)-1+u \geq 1$, contradicting lemma \ref{lemmaDistinct}. If $m \leq n$, then $\alpha(m) \preceq_\phi \alpha(n)$, and so $\alpha(n) \in \mathcal{O}^{+}_\phi(\alpha(m)) = \mathcal{O}^{+}_\phi(\phi(\alpha(n)))$. This implies $\exists u \in \mathbb{N} : \alpha(n) = \phi^{u+1}(\alpha(n))$ where $u+1 \geq 1$, contradicting again lemma \ref{lemmaDistinct}. Hence $l(n)=1$ and the relation $\alpha(n+1)=\phi(\alpha(n))$ is proved. Now consider $\widetilde{a} = \alpha(0)$. Then clearly $\mathcal{O}^{+}_\phi(\widetilde{a}) = \alpha(\mathbb{N}) = I$.
This establishes the equivalence of the last two statements. \\
Using lemma \ref{lemmaOrbPhiITotalOrder}, the last statement clearly implies the second. Let's prove that the second one implies the last one. Using lemma \ref{lemmaOrbCofImplications}, we have that all the orbits of $\phi$ are cofinite. Let $b \in I$ such that $|I \setminus \mathcal{O}^{+}_\phi(b)|$ is minimal. Suppose by way of contradiction that $\mathcal{O}^{+}_\phi(b) \neq I$ and let $c \in I \setminus \mathcal{O}^{+}_\phi(b)$. Since $\preceq_\phi$ is a total order, we have necessarily $b \in \mathcal{O}^{+}_\phi(c)$. Specifically, we have $\exists n \geq 1 : b = \phi^n(c)$. Since $c,\phi(c),\phi^2(c),\cdots$ are distinct by lemma \ref{lemmaDistinct}, we have, thus, $|I \setminus \mathcal{O}^{+}_\phi(c)| < |I \setminus \mathcal{O}^{+}_\phi(b)|$ which is a contradiction. Hence $\mathcal{O}^{+}_\phi(b) = I$. \\
The first and second statements are equivalent by corollary \ref{corollaryOrbPhiIAllOrbCof} and proposition \ref{propOrbitsInfiniteGuExistEquiv}.
\end{proof}

Note that we can also prove explicitly that 
\[ \exists a \in I : \mathcal{O}^{+}_\phi(a) = I \Rightarrow \left( \exists u : \mathcal{P}_{\omega,*}(I) \to I : \exists G : \mathcal{P}_{\omega,*}(I) \to \mathcal{P}_{\omega,*}(I) : P_2(\phi,G,u) \right). \] \\
Indeed, suppose $\exists a \in I : \mathcal{O}^{+}_\phi(a) = I$ and define
\[ n : \begin{cases} \mathcal{P}_{\omega,*}(I) &\to \mathbb{N} \\
                       I^* &\mapsto \max\{n \in \mathbb{N} : \phi^n(a) \in I^* \}
\end{cases} \]
which is well-defined because $\{a,\phi(a),\phi^2(a),\cdots\}$ are distinct (lemma \ref{lemmaDistinct}) and let 
\[ u : \begin{cases} \mathcal{P}_{\omega,*}(I) &\to I \\
                       I^* &\mapsto \phi^{n(I^*)}(a)
\end{cases} \]
and 
\[ G : \begin{cases} \mathcal{P}_{\omega,*}(I) &\to \mathcal{P}_{\omega,*}(I) \\
                       I^* &\mapsto \{a, \phi(a), \cdots, \phi^{n(I^*)}(a)\} 
\end{cases}. \]
Then $u$ and $G$ clearly satisfy the requirements.

\begin{remark}
There are many maps $\phi : \mathbb{N} \to \mathbb{N}$ such that $\exists a \in \mathbb{N} : \mathcal{O}^{+}_\phi(a) = \mathbb{N}$ (not just the successor function). They are conjugated to the successor function and can be found by choosing a bijection $\alpha : \mathbb{N} \to \mathbb{N}$ and setting $\phi = \alpha \circ succ \circ \alpha ^{-1}$. For example, there is the function
\[ \begin{cases} \mathbb{N} &\to \mathbb{N}^* \\
                                        0 &\mapsto 2 \\
                                   2i+1 &\mapsto 2i+4 \text{ for all } i \in \mathbb{N} \\
                                   2i+2 &\mapsto 2i+1 \text{ for all } i \in \mathbb{N}                
\end{cases} \]
\end{remark}

\subsection{An indivisibility property for 3-tuples $(\phi,G,u)$ satisfying $P_2(\phi,G,u)$}

3-tuples ($\phi : I \to I, G : \mathcal{P}_{\omega,*}(I) \to \mathcal{P}_{\omega,*}(I), u : \mathcal{P}_{\omega,*}(I) \to I)$ satisfying $P_2(\phi,G,u)$ enjoy the following indivisibility property with respect to composition of maps $\beta : I \to I$ satisfying $\forall I^* \in \mathcal{P}_{\omega,*}(I) : \exists a \in G(I^*) : |\mathcal{O}^{+}_\beta(a)| = +\infty$ with bijective maps $\alpha : I \to I$ satisfying $\forall I^* \in \mathcal{P}_{\omega,*}(I) : \alpha(G(I^*)) \subseteq G(I^*)$.

\begin{proposition}
\label{propIrreducibilityPhiGu}
Let $I$ be a set and ($\phi : I \to I, G : \mathcal{P}_{\omega,*}(I) \to \mathcal{P}_{\omega,*}(I), u : \mathcal{P}_{\omega,*}(I) \to I)$ a 3-tuple satisfying $P_2(\phi,G,u)$. \\
Then if $\phi = \beta \circ \alpha$ where $\beta : I \to I$ is a map satisfying $\forall I^* \in \mathcal{P}_{\omega,*}(I) : \exists a \in G(I^*) : |\mathcal{O}^{+}_\beta(a)| = +\infty$ and $\alpha : I \to I$ is a bijective map such that $\forall I^* \in \mathcal{P}_{\omega,*}(I) : \alpha(G(I^*)) \subseteq G(I^*)$, then $\alpha = id$ and $\beta = \phi$.
\end{proposition}

\begin{proof}
First, notice that for all $I^* \in \mathcal{P}_{\omega,*}(I)$, $\alpha$ bijective and $\alpha(G(I^*)) \subseteq G(I^*)$ implies that $\alpha$ induces a bijection on $G(I^*)$. \\
Let $v = \begin{cases} \mathcal{P}_{\omega,*}(I) &\to I \\ I^* &\mapsto \alpha(u(I^*)) \in G(I^*) \end{cases}$. Let $I^* \in \mathcal{P}_{\omega,*}(I)$. Since $\alpha$ induces a bijection on $G(I^*)$, it is easily seen that $\beta(G(I^*)\setminus \{v(I^*)\}) \subseteq G(I^*)$. Since $\exists a \in G(I^*) : |\mathcal{O}^{+}_\beta(a)| = +\infty$, it follows from proposition \ref{propP1} that $v(I^*) = \xi_\beta(H(I^*))$, where we have used the notation of that proposition. Specifying this relation for $I^* = \{b\}$, one has $\alpha(b) = \alpha(u(\{b\})) = v(\{b\}) = \xi_\beta(H(\{b\})$ for all $b \in I$. \\
Since $\forall I^* \in \mathcal{P}_{\omega,*}(I) : \alpha(G(I^*)) \subseteq G(I^*)$, we have $\forall b \in I : \alpha(b) \in \alpha(G(\{b\})) \subseteq G(\{b\})$, and by induction $\mathcal{O}^{+}_\alpha(b) \subseteq G(\{b\})$. Since $G(\{b\})$ is finite and $\alpha$ is bijective, we have from lemma \ref{lemmaDistinct} $\exists n_b \geq 1 : \alpha^{n_b}(b) = b$. This implies in particular that $\mathcal{O}^{+}_\beta(b) = \mathcal{O}^{+}_\beta(\alpha^{n_b}(b)) = \mathcal{O}^{+}_\beta(\xi_\beta(H(\{\alpha^{n_b-1}(b)\}))) = \bigcap_{c \in H(\{\alpha^{n_b-1}(b)\}))} \mathcal{O}^{+}_\beta(c)$ is an orbit arising as a finite intersection of infinite orbits. Thus it is infinite, which is true for all $b \in I$. \\
Now let $b \in I$. We have $|\mathcal{O}^{+}_\beta(b)|=|\mathcal{O}^{+}_\beta(\alpha(b))|=+\infty$. Let $I^* = H(\{\alpha(b)\})$. Since $\alpha(b) \in I^*$, we have $\xi_\beta(I^*) \in \mathcal{O}^{+}_\beta(\alpha(b))$. Hence $\alpha(\alpha(b)) = \xi_\beta(I^*) \in \mathcal{O}^{+}_\beta(\alpha(b))$. So $\forall b \in I : \exists k_b \in \mathbb{N} : \alpha(\alpha(b)) = \beta^{k_b}(\alpha(b))$. We have $\forall b \in I :  \alpha(\alpha(\alpha(b))) = \beta^{k_{\alpha(b)}}(\alpha(\alpha(b))) = \beta^{k_{\alpha(b)}+k_b}(\alpha(a))$. In general, $\forall b \in I : \forall n \geq 1 : \alpha^n(b) \in \mathcal{O}^{+}_\beta(\alpha(b))$. Let $b \in I$ and choose $n = n_{\alpha(b)} \geq 1$. Then $\alpha(b) = \alpha^{n_{\alpha(b)}}(\alpha(b)) \in \mathcal{O}^{+}_\beta(\alpha(\alpha(b)))$, and so $\exists m \in \mathbb{N} : \alpha(b) = \beta^m(\alpha(\alpha(b)))$. We saw previously that $\alpha(\alpha(b)) \in \mathcal{O}^{+}_\beta(\alpha(b))$, so $\exists l \in \mathbb{N} : \alpha(\alpha(b)) = \beta^l(\alpha(b))$. Combining the two results, we have $\alpha(b) = \beta^{m+l}(\alpha(b))$ and so $m=l=0$ since $|\mathcal{O}^{+}_\beta(\alpha(b))| = +\infty$. In turn, this implies that $\alpha(\alpha(b)) = \alpha(b)$ and so $\alpha(b) = b$ since $\alpha$ is a bijection. \\
Hence $\alpha = id$ and $\beta = \phi$.
\end{proof}

\section*{Acknowledgement}
The first author is financially supported by the \textit{Centre National pour la Recherche Scientifique et Technique} of Morocco.

\nocite{*}
\bibliographystyle{plain}
\bibliography{references}

\Addresses

\end{document}